\documentclass[times,final]{elsarticle}

\RequirePackage{geometry}
\usepackage{amssymb}
\usepackage{amsmath,bm}
\usepackage{amsthm}
\usepackage{color}

\newtheorem{theorem}{Theorem}[section]
\newdefinition{rmk}{Remark}[section]
\begin{document}

\begin{frontmatter}

\title{Linear, decoupled, second-order and structure-preserving scheme for Carreau fluid equations coupled with steric Poisson-Nernst-Planck model}%

\author[1]{Wenxing Zhu}
\author[2]{Mingyang Pan}
\author[1]{Dongdong He\corref{cor1}}
\cortext[cor1]{Corresponding author.}
\ead{hedongdong@cuhk.edu.cn}

\address[1]{School of Science and Engineering, The Chinese University of Hong Kong, Shenzhen, Shenzhen, Guangdong, 518172, P.R.China}
\address[2]{School of Science, Hebei University of Technology, Tianjin, 300401, P.R.China}

\begin{abstract}
In this paper, to study ionic steric effects, we present a linear, decoupled, second-order accurate in time and structure-preserving scheme with finite element approximations for Carreau fluid equations coupled with steric Poisson-Nernst-Planck (SPNP) model. The logarithmic transformation for the ion concentration is used to preserve positivity property.
To deal with the nonlinear coupling terms in fluid equation, a nonlocal auxiliary variable with respect to the free energy of SPNP equations and its associated ordinary differential equation are introduced. The obtained system is equivalent to the original system. The fully discrete scheme is proved to be mass conservative, positivity-preserving for ion concentration and energy dissipative at discrete level. 
Some numerical simulations are provided to demonstrate its stability and accuracy. Moreover, the ionic steric effects are numerically investigated.

\end{abstract}

\end{frontmatter}

\section{Introduction}
Electrokinetic flow describes the motion of electrically induced fluids. The ions are transported by advection due to fluid flow, migration under an electric potential which is generated by the distribution of charged ions and diffusion driven by concentration gradients. Conversely, the flow of fluid is driven by the electric field generated by the ions. 
The electrokinetic phenomenon has received much attention recently due to its widely promising applications including biomedical lab-on-a-chip devices~\cite{lee2000surface}, capillary electrophoresis~\cite{ghosal2006electrokinetic},  battery and fuel cell technology~\cite{nielsen2014concentration,choi2006advanced}, desalination of water~\cite{nikonenko2014desalination}, microfluidic systems~\cite{tsai2005numerical,hu2006electrokinetically}.


Electrokinetic phenomenon is usually described by the Poisson-Nernst-Planck (PNP) equations for ion transport coupled to the Navier-Stokes (NS) equations for the Newtionian fluid flow. The conventional PNP model~\cite{gajewski1986basic,brezzi1989numerical,eisenberg2007poisson,lu2010poisson,wei2012variational} treats ions as point charges and neglects the steric effects of ions arising from their finite size. While valid for bulk or diluted solvents, this assumption could be problematic in a narrow channel pore with crowded ionic populations. Within the PNP framework, there has been a growing interest in incorporating steric effects~\cite{Horng2012PNPEW,Kili2006StericEI,Lu2011PoissonNernstPlanckEF,Siddiqua2017AMP,lin_eisenberg_2013,Ding2019PositivityPF}. Currently, most studies on electrokinetic flow mainly rely on the Newtonian fluid flow~\cite{bauer2012stabilized,schmuck2009analysis,metti2016energetically,linga2020transient,pan2024linear,wang2019quasi,wang2016generalized,Xu_2014,li2024error}, which fails to capture the shear-dependent viscosity in complex fluids. To advance modeling of electrokinetic phenomenon in rheologically complex media, we aim to investigate the dynamics of ions with finite size in non-Newtonian fluids with shear-dependent viscosity. 

In this paper, we assume $\Omega$ to be a rectangular bounded domain in $\mathbb{R}^2$ with a Lipschitz continuous boundary $\partial\Omega$. The magnetic field is negligible, since the dynamic electric current is small. The fluid is assumed incompressible. Under these assumptions, the SPNP system governs the dynamics of ions with steric effects under the effect of electrical field, while the mass and momentum conservation equations describe the non-Newtonian fluid behavior. Therefore, the governing equations of Carreau fluid equations coupled with SPNP model may be written as
\begin{subequations}\label{model}
\begin{align}
&\rho\partial_t\mathbf{u} + \rho(\mathbf{u}\cdot\nabla)\mathbf{u}  - \nabla\cdot\bm{\tau} + \nabla p = -\sum\limits_{i}z_ic_i\nabla V,\label{NS} \\
&\nabla\cdot \mathbf{u}=0,\label{incompressible} \\
&\partial_tc_i + \mathbf{u}\cdot\nabla c_i = D_i\nabla\cdot(c_i\nabla g_i ),\label{concentration} \\
&g_i = \log \frac{c_i}{c_0} + \frac{z_i e}{k_BT} V + \frac{e}{k_BT}\sum\limits_{j}\omega_{ij} c_j,\label{chemicalpotential}\\
&-\nabla\cdot(\epsilon\nabla V) = \sum\limits_{i}z_ic_i,\label{Possion}
\end{align}
\end{subequations}
where $c_i$ is the ion concentration of the ith species with $i\in 1,...,N$, $g_i$ is the chemical potential with respect to $c_i$, $z_i$ is the ionic valency, $D_i$ is the diffusion constant, $W=(\omega_{ij})_{N\times N}$ is the coefficient matrix for steric interactions, $\omega_{ij}$ is a nonnegative constant depending on ion radii, $k_B$ is the Boltzmann's constant, $T$ is the absolute temperature, $e$ is the unit charge, $V$ is the electrostatic potential, $\epsilon$ is the electric permittivity, $\rho$ is the fluid density, $\mathbf{u}$ is fluid velocity, $p$ is the pressure, $\bm{\tau}=2\mu_p\mathbb{D}\mathbf{u}$ is the shear stress tensor, $\mathbb{D}\mathbf{u}=\frac{1}{2}[\nabla \mathbf{u} + (\nabla \mathbf{u})^T]$ is the deformation tensor, and the apparent viscosity $\mu_p$ is assumed to follow the Carreau model~\cite{carreau1972rheological,animasaun2017numerical,waqas2017numerical}:
\begin{align}\label{Carreau}
&\mu_p = \mu_{\infty} +(\mu_0 - \mu_{\infty})\Big( 1 + \lambda_1^2(2\mathbb{D}\mathbf{u}:\mathbb{D}\mathbf{u})\Big)^{\frac{k-1}{2}}, 
\end{align}
$\mu_0$ is the viscosity at zero shear rate, $\mu_\infty$ $(\mu_0 >\mu_\infty> 0)$ is the viscosity at infinite shear rate, $\lambda_1$ is the relaxation time constant, $k$ is the power index. It can be seen from \eqref{Carreau} that at a low shear rate, Carreau fluid behaves as a Newtonian fluid and at a high shear rate as a power law~\cite{sivakumar2006effect, zhu2024decoupled} fluid. The Carreau model is mostly used for food, beverages and also blood flow applications. In this system \eqref{model}, we assume that $W$ is positive definite and $D_i=D$.




The classical Navier–Stokes–Poisson–Nernst–Planck (NSPNP) system can hardly be solved analytically due to the coupling of the velocity, the electric potential and ion concentrations. For the two separate subproblems, there are many effective methods available. The projection method~\cite{guermond2000projection,rannacher2006chorin,guermond2006overview,guermond2009splitting} can decouple the velocity and pressure for NS equations. For the PNP equations, many efficient approaches are discussed in~\cite{prohl2009convergent,flavell2014conservative,gao2017linearized,he2016energy,shen2018scalar,he2019positivity,hu2020fully,shen2021unconditionally}. 
For NSPNP system, some numerical studies have been devoted to solving this system. Prohl et al.~\cite{Prohl2010CONVERGENTFE} constructed an implicit and nonlinear scheme which preserves the non-negativity of the ionic concentrations. They also considered a projection method without non-negativity preserving.
Metti et al.~\cite{metti2016energetically} also proposed a nonlinear, first-order and unconditionally energy stable time-stepping scheme with a logarithmic transformation of the ionic concentrations.
He et al.~\cite{he2018mixed} proposed nonlinear schemes, which preserves the positivity and satisfies energy dissipation under certain conditions and specific spatial discretization.


Compared with classical NSPNP system, the new model \eqref{model} considers non-Newtonian viscosity and ionic steric interactions. Besides the Coulomb force term coupling the electric potential and fluid velocity and the convection term coupling ionic concentration and fluid velocity, 
nonlinearity is further introduced by steric effects from the last term in \eqref{chemicalpotential} and nonlinear viscosity dependent on shear rate in \eqref{Carreau}. Developing efficient and structure-preserving algorithm is very challenging for this highly nonlinear coupled partial differential equation system \eqref{model}. This work is inspired by scalar auxiliary variable (SAV) approach to construct the structure-preserving scheme, which has received much attention recently in ~\cite{ZHOU2023108763,shen2019new,shen2018scalar,shen2021unconditionally}. Also the fully decoupled scheme can be achieved by combining with splitting method. Based on this idea, we introduce a suitable auxiliary variable and design an ordinary differential equation to deal with the nonlinear terms involved in the equations. Meanwhile, the projection method is employed to treat the fluid equations to decouple the velocity and pressure. Using logarithm transformation for the ion concentrations, we construct a linear, decoupled and second-order accurate in time scheme for the Carreau fluid equations coupled with steric PNP system which satisfies positivity preserving, mass conservative, and energy dissipative properties.

The remainder of this paper is organized as follows. In section 2, we present the dimensionless form of system \eqref{model} and reformulate it into equivalent system based on auxiliary variable approach. The energy dissipation law of the modified system is derived.
In section 3, a second-order accurate in time, linear, decoupled, and structure-preserving scheme for the reformulated system is constructed. Mass conservation, positivity for ion concentration and energy stability of the fully discrete scheme are proved. The implementation details of the proposed scheme is also presented. 
In section 4, some numerical examples are given to validate the theoretical analyses and the effctiveness of the proposed methods. 
Finally, conclusions of the present study are given in section 5.
\section{Carreau fluid equations coupled with steric PNP model}
We first introduce the governing equations in a nondimensional form, and then present three important physical properties. 

\subsection{Non-dimensionalisation}
To get a nondimensional formulation, we introduce the following dimensionless variables:
\begin{align*}
&\tilde{\bm{x}} = \frac{\bm{x}}{\hat{l}}, \quad \tilde{\mathbf{u}} = \frac{\mathbf{u}}{\hat{u}},\quad \tilde{t} = \frac{t}{\hat{l}/\hat{{u}}},\quad \tilde{c}_i = \frac{c_i}{c_0},\quad \tilde{\mu}_p=\frac{\mu_p}{\hat{\mu}},\quad \tilde{\mu}_0=\frac{\mu_0}{\hat{\mu}}\\
&\tilde{V} =\frac{V}{k_BT/e}, \quad\tilde{p} = \frac{p}{\rho \hat{{u}}^2}, \quad \tilde{\omega}_{ij}=\frac{\omega_{ij}c_0e}{k_BT},\quad \tilde{\mu}_{\infty}=\frac{\mu_{\infty}}{\hat{\mu}}, \quad\tilde{\lambda}_1=\lambda_1\frac{\hat{u}}{\hat{l}}.
\end{align*}
For clarity, we omit all the tildes of the dimensionless variables. Then the govern equations \eqref{model} in dimensionless form become:
\begin{subequations}\label{nondimensional-model}
\begin{align}
&\partial_t\mathbf{u} + (\mathbf{u}\cdot\nabla)\mathbf{u}  - \frac{1}{Re} \nabla\cdot(2\mu_p\mathbb{D}\mathbf{u}) + \nabla p = -Co\sum\limits_{i}z_ic_i\nabla V,\label{non-NS} \\
&\nabla\cdot \mathbf{u}=0,\label{non-incompressible} \\
&\partial_tc_i + \nabla\cdot (\mathbf{u}c_i) = \frac{1}{Pe}\nabla\cdot(c_i\nabla g_i),\label{non-concentration} \\
&g_i=\log c_i + z_iV + \sum\limits_{j}\omega_{ij}c_j,\label{non-potential}\\
&-\lambda\Delta V = \sum\limits_{i}z_ic_i,\label{non-Possion}\\
&\mu_p = \mu_{\infty} +(\mu_0 - \mu_{\infty})\Big( 1 + \lambda_1^2(2\mathbb{D}\mathbf{u}:\mathbb{D}\mathbf{u})\Big)^{\frac{k-1}{2}},
\end{align}
\end{subequations}
where the nondimensional numbers are defined as follows
\begin{align*}
Re=\frac{ \rho\hat{u} \hat{l}}{\hat{\mu}},\quad Co=\frac{c_0k_BT}{\rho \hat{u}^2e},\quad Pe=\frac{\hat{l}\hat{u}}{D},\quad \lambda=\frac{\epsilon k_BT}{\hat{l}^2c_0e}.
\end{align*}
Here, $Re$ is the Reynolds number, $Co$ is the Coulomb-driven number, $Pe$ is the P\'{e}lect number, and $\lambda$ is the ratio of Debye length to the characteristic length. The initial conditions and boundary conditions are given by
\begin{align}
&c_i|_{t=0}=c_{i0}, \quad -\lambda\Delta V|_{t=0} = \sum\limits_{i}z_ic_{i0}, \quad \bm{u}|_{t=0}=\bm{u}_0,\label{eq:IC}\\
&\frac{\partial c_i}{\partial \mathbf{n}}|_{\partial\Omega}=0,\quad \frac{\partial V}{\partial \mathbf{n}}|_{\partial\Omega}=0,  \quad \bm{u}|_{\partial\Omega}=\bm{0},\label{BC}
\end{align}
where $\mathbf{n}$ is the unit outward normal on the boundary $\partial\Omega$.
Note that the electrostatic potential $V$ and pressure $p$ are determined up to a constant, we select unique solutions $V$, $p$ by requiring
\begin{align*}
\int_{\Omega} V d\mathbf{x}=0,\quad \int_{\Omega} p d\mathbf{x}=0.
\end{align*}
The total free energy of \eqref{nondimensional-model} is defined as:
\begin{align}\label{Energy}
E = E_{\mathbf{u}} + E_{V} + E_{ent}+ E_{ster},
\end{align}
where kinetic energy $E_{\mathbf{u}}$, electric field energy $E_V$, entropic contribution $E_{ent}$, and steric interaction energy $E_{ster}$ are given by
\begin{align*}
&E_{\mathbf{u}} = \int_{\Omega} \frac{1}{2}|\mathbf{u}|^2 d\mathbf{x},\\
&E_{V} = \int_{\Omega} \frac{1}{2}\lambda Co|\nabla V|^2 d\mathbf{x},\\
&E_{ent} = \int_{\Omega} Co\sum\limits_{i}c_i(\log c_i-1) d\mathbf{x},\\
&E_{ster} = \int_{\Omega} \frac{1}{2}Co\sum\limits_{i}\sum\limits_{j}\omega_{ij}c_ic_j d\mathbf{x}.
\end{align*}

\subsection{Basic properties}

We now present a few physical properties of the system \eqref{nondimensional-model}, which play important role in designing numerical schemes.
\begin{enumerate}
  \item Positivity: Suppose that $c_{i}(\mathbf{x},0)>0, \forall \mathbf{x}\in\Omega$, then
  \begin{align*}
  c_{i}(\mathbf{x},t)>0, \quad \forall t>0,\mathbf{x}\in\Omega.
  \end{align*}
  The term $\log c_i$ in \eqref{non-concentration} must require $c_i$ is always positive.
  \item Mass conservation:
  \begin{align*}
  \int_{\Omega} c_{i}(\mathbf{x},t) d\mathbf{x} = \int_{\Omega}  c_{i}(\mathbf{x},0) d\mathbf{x}, \quad \forall t>0.
  \end{align*}
  The mass conservation can be obtained by integrating \eqref{non-concentration} over $\Omega$, using the integration by parts and applying  the boundary conditions \eqref{BC}.
  \item Energy dissipation:
  \begin{align*}
\frac{dE}{dt}\leq 0.
  \end{align*}
The energy dissipation law for the system \eqref{nondimensional-model} can be obtained by the following process. Taking the $L^2$ inner product of \eqref{non-NS} with $\mathbf{u}$, using \eqref{non-incompressible} and integration by parts can arrive at
\begin{equation}\label{NS1}
\frac{d}{dt}\int_{\Omega}\frac{1}{2}|\mathbf{u}|^2d\mathbf{x} = ( \partial_t\mathbf{u},\mathbf{u}) = -\frac{1}{Re}\|\sqrt{2\mu_p}\mathbb{D}\mathbf{u}\|_F^2 - Co\left(\sum\limits_{i}z_ic_i\nabla V,\mathbf{u}\right),
\end{equation}
where $\|\cdot\|_F$ is Frobenius norm. Differentiating \eqref{non-Possion} with respect to time $t$, taking the $L^2$ inner product with $CoV$ and integrating by parts, we obtain
\begin{align}\label{Possion1}
\frac{d}{dt}\int_{\Omega} \frac{Co}{2}\lambda|\nabla V|^2 d\mathbf{x}= 
-Co\Big((\nabla\cdot(\lambda\nabla V))_t,V\Big) = Co(\sum\limits_{i}z_i\partial_tc_i,V).
\end{align}
Taking the $L^2$ inner product of \eqref{non-concentration} with $Cog_i$, using \eqref{non-potential}, \eqref{non-incompressible} and integration by parts, taking the summation for $i$ to get
\begin{align}\label{concentration1}
&\frac{d}{dt}\left( \int_{\Omega}\sum\limits_{i}Coc_i(\log c_i-1)d\mathbf{x} + \int_{\Omega}\frac{Co}{2}\sum\limits_{i}\sum\limits_{j}\omega_{ij}c_ic_j d\mathbf{x} \right)+ \sum\limits_{i}Co(z_iV,\partial_tc_i)\nonumber\\
& = - \sum\limits_{i}Co\int_{\Omega}\nabla\cdot(\mathbf{u} c_i) (\log c_i + z_iV + \sum\limits_{j}\omega_{ij}c_j)  d\mathbf{x} -\frac{Co}{Pe}\sum\limits_{i}\int_{\Omega}  c_i|\nabla g_i|^2 d\mathbf{x}\nonumber\\
& = \sum\limits_{i}Co\left(\int_{\Omega} \mathbf{u}\cdot\nabla c_i d\mathbf{x} + \int_{\Omega} z_i c_i\mathbf{u}\cdot \nabla V d\mathbf{x}
+\int_{\Omega} \mathbf{u}\cdot\nabla(\frac{1}{2}\sum\limits_{j}\omega_{ij}c_ic_j) d\mathbf{x} -\frac{1}{Pe}\int_{\Omega}  c_i|\nabla g_i|^2 d\mathbf{x}\right)\nonumber\\
&=\int_{\Omega} \sum\limits_{i}Coz_i c_i\mathbf{u}\cdot \nabla V d\mathbf{x} -\frac{Co}{Pe}\sum\limits_{i}\int_{\Omega} c_i|\nabla g_i|^2 d\mathbf{x}.
\end{align}
Combining \eqref{Possion1} and \eqref{concentration1}, we have
\begin{align}\label{PNP-Energy}
\frac{d}{dt}(E_{ent}+ E_{ster} + E_{V}) = -\frac{Co}{Pe}\sum\limits_{i}\int_{\Omega} c_i|\nabla g_i|^2 d\mathbf{x}  + \int_{\Omega} \sum\limits_{i}Coz_i c_i\mathbf{u}\cdot \nabla V d\mathbf{x},
\end{align}
From \eqref{NS1} and \eqref{PNP-Energy}, the obtained energy law reads as follows:
\begin{align}\label{Energy1}
\frac{dE}{dt} = -\frac{Co}{Pe}\sum\limits_{i}\int_{\Omega} c_i|\nabla g_i|^2 d\mathbf{x} -\frac{1}{Re}\|\sqrt{2\mu_p}\mathbb{D}\mathbf{u}\|_F^2 \leq0.
\end{align}
\end{enumerate}

\subsection{Reformulation}
To preserve the positivity of the ion concentration, we consider the logarithm transformation $c_i=\exp(\sigma_i)$. Then the evolution of $\sigma_i$ is given by
\begin{align}
&\partial_t\sigma_i + \mathbf{u}\cdot\nabla \sigma_i =
\frac{1}{Pe}\Big[|\nabla \sigma_i|^2 + \Delta \sigma_i + z_i\Delta V + z_i\nabla\sigma_i\cdot\nabla V \nonumber\\
&+\sum\limits_{j}\omega_{ij}\nabla\sigma_i\cdot\nabla \sigma_j c_j + \sum\limits_{j}\omega_{ij} |\nabla \sigma_j|^2c_j + \sum\limits_{j}\omega_{ij}\Delta \sigma_j c_j \Big].
\end{align}
Next, to construct a decoupling and energy stable scheme, we introduce a nonlocal variable $r(t)$ such that
\begin{align}\label{r-equation}
r(t) = \sqrt{E_{SPNP} + B},
\end{align}
where $E_{SPNP}=E_{ent}+ E_{ster} + E_{V} $, $B$ is a positive constant to guarantee the radicand is always positive. Note that $E_{V}$, $E_{ent}$ and $E_{ster}$ are convex functions, we can always find such a constant $B$ since the summation of $E_{V}$, $E_{ent}$ and $E_{ster}$ are bounded from below. Then we  define a nonlocal variable $\xi(t)$ such that
\begin{align*}
\xi(t)=\frac{r}{\sqrt{E_{SPNP} + B}}.
\end{align*}
It is easy to see that $\xi(t)=1$. Differentiating \eqref{r-equation} with respect to time $t$, using \eqref{PNP-Energy}, adding the zero-valued term  $\int_{\Omega}(\mathbf{u}\cdot\nabla)\mathbf{u}\cdot \mathbf{u} d\mathbf{x}$ and multiplying by the factor $\xi(t)$, the associated ordinary differential equation is
\begin{align}\label{dr-equation}
\frac{dr}{dt}= \frac{1}{2\sqrt{E_{SPNP} + B}}\Big(- \frac{Co}{Pe}\xi\sum\limits_{i}\int_{\Omega}  c_i|\nabla g_i|^2 d\mathbf{x}
+ \int_{\Omega} \sum\limits_{i}Coz_i c_i\mathbf{u}\cdot \nabla V d\mathbf{x} + \int_{\Omega}(\mathbf{u}\cdot\nabla)\mathbf{u}\cdot \mathbf{u} d\mathbf{x} \Big).
\end{align}
Then by combining the nonlocal variables $r(t)$, $\xi(t)$ and the evolution \eqref{dr-equation}, the system \eqref{nondimensional-model} is reformulated to the following form:
\begin{subequations}\label{equivalent-model}
\begin{align}
& \partial_t\mathbf{u} + \xi(\mathbf{u}\cdot\nabla)\mathbf{u} - \frac{1}{Re}\nabla\cdot(2\mu_p\mathbb{D}\mathbf{u})+ \nabla p = -Co\xi\sum\limits_{i}z_ic_i\nabla V,\label{equivalent-NS} \\
&\nabla\cdot \mathbf{u}=0,\label{equivalent-incompressible} \\
&c_i=\exp(\sigma_i),\label{equivalent-concentration}\\
&\partial_t\sigma_i + \mathbf{u}\cdot\nabla \sigma_i =
\frac{1}{Pe}\Big[|\nabla \sigma_i|^2 + \Delta \sigma_i + z_i\Delta V + z_i\nabla\sigma_i\cdot\nabla V \nonumber\\
&+\sum\limits_{j}\omega_{ij}\nabla\sigma_i\cdot\nabla \sigma_j c_j + \sum\limits_{j}\omega_{ij} |\nabla \sigma_j|^2c_j + \sum\limits_{j}\omega_{ij}\Delta \sigma_j c_j \Big], \label{equivalent-logconcentration}\\
&- \lambda\Delta\bar{V} = \sum\limits_{i}z_ic_i,\label{equivalent-Possion}\\
&V = \xi\bar{V},\\
&\frac{dr}{dt} = \frac{1}{2\sqrt{E_{SPNP} + B}}\Big(-\frac{Co}{Pe}\xi\sum\limits_{i}\int_{\Omega} c_i|\nabla g_i|^2 d\mathbf{x} \nonumber\\
&+ \int_{\Omega} \sum\limits_{i}Coz_i c_i\mathbf{u}\cdot \nabla V d\mathbf{x} + \int_{\Omega}(\mathbf{u}\cdot\nabla)\mathbf{u}\cdot \mathbf{u} d\mathbf{x} \Big),\label{equivalent-r}\\
&\mu_p = \mu_{\infty} +(\mu_0 - \mu_{\infty})\Big( 1 + \lambda_1^2(2\mathbb{D}\mathbf{u}:\mathbb{D}\mathbf{u})\Big)^{\frac{k-1}{2}}.
\end{align}
\end{subequations}
The transformed system \eqref{equivalent-model} satisfies the following initial conditions:
\begin{align*}
&c_i|_{t=0}=c_{i0}, \quad -\lambda\Delta V|_{t=0} = \sum\limits_{i}z_ic_{i0}, \quad \bm{u}|_{t=0}=\bm{u}_0,\\
&r|_{t=0}=\sqrt{ \int_{\Omega} \frac{1}{2}\lambda Co|\nabla V_0|^2 d\mathbf{x} + Co\int_{\Omega} \sum\limits_{i}c_{i0}(\log c_{i0}-1) d\mathbf{x} + \int_{\Omega} \frac{1}{2}Co\sum\limits_{i}\sum\limits_{j}\omega_{ij}c_{i0}c_{j0} d\mathbf{x} + B}.
\end{align*}
Note that the new system \eqref{equivalent-model} is equivalent to the original PDE system \eqref{nondimensional-model}. Also, the new system \eqref{equivalent-model} follows an energy dissipative law. Taking the $L^2$ inner product of \eqref{equivalent-NS} with $\mathbf{u}$, using \eqref{equivalent-incompressible} and integration by parts, we obtain
\begin{align}\label{0-1}
\frac{d}{dt}\int_{\Omega}\frac{1}{2}|\mathbf{u}|^2d\mathbf{x} = -\frac{1}{Re}\|\sqrt{2\mu_p}\mathbb{D}\mathbf{u}\|_F^2 - Co\xi\left(\sum\limits_{i}z_ic_i\nabla V,\mathbf{u}\right).
\end{align}
By multiplying \eqref{equivalent-r} with $2r$, we obtain
\begin{align}\label{0-2}
\frac{d r^2}{dt} = -\frac{Co}{Pe}\left|\xi\right|^2\sum\limits_{i}\int_{\Omega} c_i|\nabla g_i|^2 d\mathbf{x} + Co\xi \int_{\Omega} \sum\limits_{i}z_i c_i\mathbf{u}\cdot \nabla V d\mathbf{x}.
\end{align}
Combining \eqref{0-1} and \eqref{0-2}, the obtained energy law in the new form can read as:
\begin{align}\label{energylaw}
\frac{d}{dt}\left(\int_{\Omega}\frac{1}{2}|\mathbf{u}|^2d\mathbf{x} + r^2-B\right)= -\frac{Co}{Pe}\left|\xi\right|^2\sum\limits_{i}\int_{\Omega} c_i|\nabla g_i|^2 d\mathbf{x} - \frac{1}{Re}\|\sqrt{2\mu_p}\mathbb{D}\mathbf{u}\|_F^2.
\end{align}
\section{Full discrete scheme}
Let $0<h<1$ denote the mesh size and $\mathcal{T}_{h}$ be a uniform partition of $\bar{\Omega}=\cup_{K\in \mathcal{T}_{h} }{K}$ into non-overlapping triangles. Given a $\mathcal{T}_{h}$, we consider the following finite element space
\begin{align*}
&\bm{X}_{h}=\left\{\mathbf{v}_h\in H_0^1(\Omega)^2\cap C^0(\bar{\Omega})^2
: {v_h}_{|K}\in P_{l_1}(K)^2, \forall K\in \mathcal{T}_{h}\right\},\\
&M_{h}=\left\{q_{h}\in L_0^2(\Omega)\cap C^0(\bar{\Omega}): {q_h}_{|K}\in P_{l_2}(K), \forall K\in \mathcal{T}_{h}\right\},\\
&Q_{h}=\left\{\varphi_{h}\in H^1(\Omega)\cap C^0(\bar{\Omega}): {\varphi_h}_{|K}\in P_{l_3}(K), \forall K\in \mathcal{T}_{h}\right\},\\
&S_{h}=Q_h \cap L_0^2(\Omega),
\end{align*}
where $P_{l_m}$ is the space of piecewise polynomials of total degree no more than $l_m$, $L_{0}^2(\Omega)=\left\{q\in L^2(\Omega) : \int_{\Omega}qdx=0\right\}$ and $H_{0}^1(\Omega)=\left\{s\in H^1(\Omega) : s=0 \ \ \mbox{on}\ \ \partial\Omega\right\}$. 
Additionally, assume that the finite element spaces $(\bm{X}_{h},M_{h})$ satisfy the discrete inf-sup inequality \cite{girault1979analysis}: for each $q_h\in M_h$, there exists a positive constant $\beta>0$ such that
\begin{equation*}
\sup\limits_{\mathbf{v}_h\in \bm{X}_{h},\mathbf{v}_h\neq0}\frac{(q_h,\nabla\cdot \mathbf{v}_{h})}{\|\nabla \mathbf{v}_{h}\|_{0}}\geq\beta\|q_h\|_{0}.
\end{equation*}
To discrete \eqref{equivalent-model} in space by using finite element method, we first write variational formulations for \eqref{equivalent-model} which reads as: find $(\sigma,c_i, V, \bar{V},\bm{u}, p)\in H^1(\Omega)\times H^1(\Omega)\times (H^1(\Omega)\cap L_0^2(\Omega))\times (H^1(\Omega)\cap L_0^2(\Omega))  \times H_0^1(\Omega)^2\times L_0^2(\Omega)  $ such that for all $(\eta_i,\phi, \bm{v}, q)\in  H^1(\Omega)\times H^1(\Omega)\times H_0^1(\Omega)^2\times L_0^2(\Omega) $,
\begin{subequations}\label{equivalent-model}
\begin{align}
&( \partial_t\bm{u},\bm{v}) + ( \xi(\bm{u}\cdot\nabla)\bm{u},\bm{v}) + \frac{1}{Re}(2\mu_p\mathbb{D}\bm{u},\mathbb{D}\bm{v})+ ( \nabla p,\bm{v}) = -Co ( \sum\limits_{i}z_ic_i\nabla V,\bm{v})\label{Var-equivalent-NSS} \\
&(\nabla\cdot \bm{u},q)=0,\label{Var-equivalent-incompressible} \\
&c_i=\exp(\eta_i),\label{Var-equivalent-concentration}\\
&(\partial_t\sigma_i,\eta_{i}) + (\bm{u}\cdot\nabla \sigma_i,\eta_{i}) =
\frac{1}{Pe}\Big[(|\nabla \sigma_i|^2,\eta_{i}) - ( \nabla \sigma_i,\nabla \eta_{i}) - z_i(\nabla V,\nabla \eta_{i})  \nonumber\\
&+ z_i(\nabla\sigma_i\cdot\nabla V , \eta_{i}) +\sum\limits_{j}(\omega_{ij}\nabla\sigma_i\cdot\nabla \sigma_j c_j, \eta_{i}) - \sum\limits_{j}\omega_{ij}\nabla \sigma_j c_j,\nabla  \eta_{i}) \Big], \label{Var-equivalent-logconcentration}\\
&\lambda (\nabla\bar{V},\nabla \phi) = \sum\limits_{i}(z_ic_i,\phi),\label{Var-equivalent-Possion}\\
&\xi=\frac{r }{\sqrt{E_{SPNP} + B}},\\
&V = \xi\bar{V},\\
&\frac{dr}{dt} = \frac{1}{2\sqrt{E_{SPNP} + B}}\Big(-\frac{Co}{Pe}\xi\sum\limits_{i}\int_{\Omega} c_i|\nabla g_i|^2 d\mathbf{x} \nonumber\\
&\qquad + \int_{\Omega} \sum\limits_{i}Coz_i c_i\bm{u}\cdot \nabla V d\mathbf{x} + \int_{\Omega}(\bm{u}\cdot\nabla)\bm{u}\cdot \bm{u} d\mathbf{x} \Big),\label{Var-equivalent-r}\\
&\mu_p = \mu_{\infty} +(\mu_0 - \mu_{\infty})\Big( 1 + \lambda^2(2\mathbb{D}\bm{u}:\mathbb{D}\bm{u})\Big)^{\frac{k-1}{2}}.
\end{align}
\end{subequations}

\subsection{Numerical scheme}

We discretize the system \eqref{equivalent-model} in space by finite element method and in time by second-order backward differentiation formula, leading to the fully discrete second-order scheme: find 
$(\sigma_{ih}^{n+1},\bar{c}_{ih}^{n+1},c_{ih}^{n+1}, \bar{V}_{h}^{n+1},
V_{h}^{n+1}, \bm{u}_{h}^{n+1},\\p_{h}^{n+1}) \in (Q_h,Q_h,Q_h,S_h,S_h,\bm{X}_h,M_h)$, such that for all
 $(\eta_{ih},\phi_h,\bm{v}_{h},q_{h})\in (Q_h,Q_h,\bm{X}_h,M_h)$
\begin{subequations}\label{numerical-model}
\begin{align}
& (\frac{3\tilde{\mathbf{u}}^{n+1}_{h} -4\mathbf{u}^{n}_{h} +\mathbf{u}^{n-1}_{h}}{2\Delta t},\mathbf{v}_h) + \xi^{n+1}((\mathbf{u}^{*}_{h}\cdot\nabla)\mathbf{u}^{*}_{h},\mathbf{v}_h) + \frac{1}{Re} (2\mu_{ph}^*\mathbb{D}\tilde{\mathbf{u}}^{n+1}_{h},\mathbb{D}\mathbf{v}_h) \nonumber\\
&+ (\nabla p^{n}_{h},\mathbf{v}_h)= -Co\xi^{n+1}(\sum\limits_{i}z_ic_{ih}^{n+1}\nabla \bar{V}_{h}^{n+1},\mathbf{v}_h),\label{discrete-numerical-NS} \\
&(\nabla \psi^{n+1}_h,\nabla q_h) = \frac{3}{2\Delta t}(\tilde{\mathbf{u}}_{h}^{n+1},\nabla q_h),\label{discrete-numerical-psi}\\
&\mathbf{u}_h^{n+1} = \tilde{\mathbf{u}}_h^{n+1} - \frac{2\Delta t}{3}\nabla \psi_h^{n+1},\label{discrete-numerical-correction-u}\\
&p^{n+1}_h=  \psi^{n+1}_h + p_h^{n},\label{discrete-numerical-correction-p}\\
&(\frac{ 3\sigma_{ih}^{n+1} - 4\sigma_{ih}^{n} + \sigma_{ih}^{n-1} }{2\Delta t},\eta_{ih}) + (\mathbf{u}_h^{*}\cdot\nabla \sigma_{ih}^{n+1}, \eta_{ih})+ (\nabla \sigma_{ih}^{n+1},\nabla\eta_{ih})\nonumber \\
&= \frac{1}{Pe}\Big((\nabla \sigma_{ih}^{*}\cdot\nabla \sigma_{ih}^{n+1}, \eta_{ih}) - z_i(\nabla V_h^{*},\nabla\eta_i) + (z_i\nabla\sigma_{ih}^{n+1}\cdot\nabla V_h^{*}, \eta_{ih})\nonumber\\
&+ (\sum\limits_{j}\omega_{ij}\nabla\sigma_{ih}^{n+1}\cdot\nabla \sigma_{jh}^{*} c_{jh}^{*},\eta_{ih} ) -  (\omega_{ij}\nabla \sigma_{jh}^{*} c_{jh}^{*},\nabla\eta_{ih}) -  (\omega_{ii}\nabla \sigma_{ih}^{n+1} c_{ih}^{*},\nabla\eta_{ih}) \Big),\label{discrete-numerical-logconcentration}\\
&\bar{c}_{ih}^{n+1}=\exp(\sigma_{ih}^{n+1}),\label{discrete-numerical-concentration}\\
&c_{ih}^{n+1}=\frac{( c_{ih}^{n},1) }{( \bar{c}_{ih}^{n+1},1)}\bar{c}_{ih}^{n+1},\label{discrete-numerical-concentration-mass}\\
&(\lambda\nabla \bar{V}_{h}^{n+1}, \nabla\phi_h)= \sum\limits_{i}(z_ic_{ih}^{n+1}, \phi_h),\label{discrete-numerical-Possion}\\
&V^{n+1}_h = \xi^{n+1}\bar{V}^{n+1},\\
&\xi^{n+1}=\frac{r^{n+1}}{\sqrt{\bar{E}_{SPNP}^{n+1} + B}},\\
&\frac{3r^{n+1}-4r^{n}+r^{n-1}}{2\Delta t} = \frac{1}{2\sqrt{\bar{E}_{SPNP}^{n+1} + B}}\Big( -\frac{Co}{Pe}\xi^{n+1}\sum\limits_{i}\|\sqrt{ c_{ih}^{n+1}}\nabla g_{ih}^{n+1}\|^2 \nonumber\\
&+Co (\sum\limits_{i}z_i c_{ih}^{n+1}\tilde{\mathbf{u}}^{n+1}_{h} ,\nabla \bar{V}_h^{n+1} ) + ((\mathbf{u}_{h}^{*}\cdot\nabla)\mathbf{u}_{h}^{*}, \tilde{\mathbf{u}}^{n+1}_{h}    ) \Big),\label{discrete-numerical-r}\\
&\mu_{ph}^{n+1} = \mu_{\infty} +(\mu_0 - \mu_{\infty})\Big( 1 + \lambda_1^2(2\mathbb{D}\mathbf{u}^{n+1}_{h}:\mathbb{D}\mathbf{u}^{n+1}_{h})\Big)^{\frac{k-1}{2}}.
\end{align}
\end{subequations}
where
\begin{align*}
&\mathbf{u}^{*}_{h} = 2\mathbf{u}^{n}_{h}-\mathbf{u}^{n-1}_{h}, \quad \bar{g}_{ih}^{n+1} = \log c_{ih}^{n+1} + z_i\bar{V}_h^{n+1} + \sum\limits_{j}\omega_{ij}c_{jh}^{n+1}, \\
& \sigma_{ih}^{*} = 2\sigma_{ih}^{n} - \sigma_{ih}^{n-1},\quad c_{ih}^{*} = 2c_{ih}^{n} - c_{ih}^{n-1},\quad \mu_{ph}^*=2\mu_{ph}^{n}-\mu_{ph}^{n-1},\\
&\bar{E}_{SPNP}^{n+1}=  \frac{\lambda Co}{2}\|\nabla \bar{V}_{h}^{n+1}\|^2  + Co \sum\limits_{i}(c_{ih}^{n+1},\log c_{ih}^{n+1}-1)  +  \frac{Co}{2}\sum\limits_{i}\sum\limits_{j}\omega_{ij}(c_{ih}^{n+1},c_{jh}^{n+1}) .
\end{align*}

\begin{theorem}
The fully discrete scheme \eqref{numerical-model} satisfies the following properties:
\begin{enumerate}
  \item Positivity preservation: $c_{ih}^{n+1} >0$.
  \item Mass conservation:
  \begin{align*}
   (c_{ih}^{n},1)  =   (c_{ih}^{n+1},1) .
  \end{align*}
  \item Energy dissipation:
  \begin{align}\label{discrete-energy-law}
\frac{E^{n+1}_h - E^{n}_h}{\Delta t} \leq - \frac{1}{Re} \|\sqrt{2\mu_{ph}^{*}}\mathbb{D}\tilde{\mathbf{u}}^{n+1}_h\|_F^2 - |\xi^{n+1}|^2\frac{Co}{Pe}\sum\limits_{i} \|\sqrt{ c_{ih}^{n+1}}\nabla \bar{g}_{ih}^{n+1}\|^2,
\end{align}
where
\begin{align*}
E^{n+1}_h = \frac{1}{2}\left(\frac{1}{2}\|\mathbf{u}^{n+1}_h\|^2+\frac{1}{2}\|2\mathbf{u}^{n+1}_h-\mathbf{u}^{n}_h\|^2\right) + \frac{\Delta t^2}{3}\|\nabla p_h^{n+1}\|^2 + \frac{1}{2}(|r^{n+1}|^2+|2r^{n+1}-r^{n}|^2).
\end{align*}
\end{enumerate}
\end{theorem}
\begin{proof}
First, we prove the positivity preserving. It is easy to see that $\bar{c}_{ih}^{n+1} >0$ from \eqref{discrete-numerical-concentration} and $c_{ih}^{n} >0$ from positivity of the ion concentration at the previous step. Thus we can arrive at
\begin{align*}
(\bar{c}_{ih}^{n+1},1) >0, \quad   ({c}_{ih}^{n} ,1)>0.
\end{align*}
By \eqref{discrete-numerical-concentration-mass}, we have $c_{ih}^{n+1} >0$.

Next, the mass conservation can be obtained by taking the $L^2$ inner product   with the constant function 1 on both sides of the equation \eqref{discrete-numerical-concentration-mass}.

It remains to prove the energy dissipation. Taking $\mathbf{v}_h=2\Delta t\widetilde{\mathbf{u}}_h^{n+1}$ in \eqref{discrete-numerical-NS}, we have
\begin{align}\label{discrete-numerical-NS1}
&(3\tilde{\mathbf{u}}^{n+1}_{h} -4\mathbf{u}^{n}_{h} +\mathbf{u}^{n-1}_{h},\tilde{\mathbf{u}}_h^{n+1}) + 2\Delta t\xi^{n+1}((\mathbf{u}^{*}_{h}\cdot\nabla)\mathbf{u}^{*}_{h},\tilde{\mathbf{u}}_h^{n+1} ) + 2\Delta t(\nabla p^{n}_{h},\tilde{\mathbf{u}}_h^{n+1})\nonumber\\
&= -2\Delta tCo\xi^{n+1}(\sum\limits_{i}z_ic_{ih}^{n+1}\nabla \bar{V}_{h}^{n+1},\tilde{\mathbf{u}}_h^{n+1})- \frac{2\Delta t}{Re} \|\sqrt{2\mu_{ph}^{*}}\mathbb{D}\tilde{\mathbf{u}}^{n+1}_h\|_F^2.
\end{align}
From \eqref{discrete-numerical-psi} - \eqref{discrete-numerical-correction-p}, we have
\begin{align}%
&(\mathbf{u}^{n+1}_h,\nabla q_h)=0,\label{correction-1}\\
&(\mathbf{u}^{n+1}_h,\tilde{\mathbf{u}}^{n+1}_h-\mathbf{u}^{n+1}_h)=\left(\mathbf{u}^{n+1}_h,\frac{2\Delta t }{3}\nabla(p_h^{n+1}-p_h^{n})\right)=0.\label{correction-2}
\end{align}
It follows from \eqref{discrete-numerical-psi} and \eqref{discrete-numerical-correction-p} that
\begin{align*}
\frac{\sqrt{3}}{\sqrt{2}}\mathbf{u}_h^{n+1} + \frac{\sqrt{2}}{\sqrt{3} }\Delta t\nabla p_h^{n+1} =  \frac{\sqrt{3}}{\sqrt{2}}\tilde{\mathbf{u}}_h^{n+1} + \frac{\sqrt{2}}{\sqrt{3}} \Delta t \nabla p_h^{n}.
\end{align*}
By taking the discrete inner product of the above equality with itself on both sides and using \eqref{correction-1}, we get
\begin{align}\label{discrete-numerical-p-1}
\frac{3}{2}\|\mathbf{u}_h^{n+1}\|^2 + \frac{2 }{3}(\Delta t)^2\|\nabla p_h^{n+1}\|^2 =  \frac{3}{2}\|\tilde{\mathbf{u}}_h^{n+1}\|^2 + \frac{2 }{3} (\Delta t)^2\|\nabla p_h^{n}\|^2 + 2\Delta t(\nabla p_h^{n}, \tilde{\mathbf{u}}_h^{n+1}).
\end{align}
Using \eqref{correction-2} and the following identity
\begin{align}\label{identity}
&2(a-b)a=|a|^2-|b|^2+|a-b|^2,\\
&2(3a-4b+c)a=|a|^2+|2a-b|^2-|b|^2-|2b-c|^2+|a-2b+c|^2,
\end{align}
we have
\begin{align}\label{discrete-numerical-NS2}
&(3\tilde{\mathbf{u}}_h^{n+1}-4 \mathbf{u}_h^n+\mathbf{u}_h^{n-1}, \tilde{\mathbf{u}}_h^{n+1}) = \Big(3 (\tilde{\mathbf{u}}_h^{n+1}-\mathbf{u}_h^{n+1}) + 3\mathbf{u}_h^{n+1} -4 \mathbf{u}_h^n+\mathbf{u}_h^{n-1}, \tilde{\mathbf{u}}_h^{n+1}\Big) \nonumber\\
&= 3(\tilde{\mathbf{u}}_h^{n+1}- \mathbf{u}_h^{n+1} ,\tilde{\mathbf{u}}_h^{n+1}) + (3 {\mathbf{u}}_h^{n+1}-4 \mathbf{u}_h^n+\mathbf{u}_h^{n-1}, \mathbf{u}_h^{n+1}) + (3 {\mathbf{u}}_h^{n+1}-4 \mathbf{u}_h^n+\mathbf{u}_h^{n-1}, \tilde{\mathbf{u}}_h^{n+1}-\mathbf{u}_h^{n+1} )  \nonumber\\
& =\frac{3}{2}\Big(\|\tilde{\mathbf{u}}_h^{n+1} \|^2 - \|\mathbf{u}_h^{n+1} \|^2 + \| \tilde{\mathbf{u}}_h^{n+1}-\mathbf{u}_h^{n+1} \|^2 \Big) + \frac{1}{2} \Big(\|\mathbf{u}_h^{n+1}\|^2+\|2 \mathbf{u}_h^{n+1}-\mathbf{u}_h^n\|^2-\|\mathbf{u}_h^n\|^2 \nonumber\\
& \quad - \|2 \mathbf{u}_h^n-\mathbf{u}_h^{n-1} \|^2 + \|\mathbf{u}_h^{n+1}-2 \mathbf{u}_h^n+\mathbf{u}_h^{n-1} \|^2 \Big)
\end{align}
By multiplying \eqref{discrete-numerical-r} with $4\Delta t r^{n+1}$ and using \eqref{identity}, we obtain
\begin{align}\label{discrete-numerical-r1}
&|r^{n+1}|^2+|2r^{n+1}-r^{n}|^2-|r^{n}|^2-|2r^{n}-r^{n-1}|^2+|r^{n+1}-2r^{n}+r^{n-1}|^2 \nonumber\\
&=2\Delta t \xi^{n+1}\Big( -\frac{Co}{Pe}\xi^{n+1}\sum\limits_{i}\|\sqrt{ c_{ih}^{n+1}}\nabla g_{ih}^{n+1}\|^2    +Co (\sum\limits_{i}z_i c_{ih}^{n+1}\tilde{\mathbf{u}}^{n+1}_{h} ,\nabla \bar{V}_h^{n+1} ) + ((\mathbf{u}_{h}^{*}\cdot\nabla)\mathbf{u}_{h}^{*}, \tilde{\mathbf{u}}^{n+1}_{h} ) \Big).
\end{align}
Combining \eqref{discrete-numerical-NS1},\eqref{discrete-numerical-p-1},\eqref{discrete-numerical-NS2} and \eqref{discrete-numerical-r1}, we arrive at
\begin{align}\label{discrete-energy-law1}
&\frac{1}{2}\left(\|\mathbf{u}_h^{n+1}\|^2+\|2\mathbf{u}_h^{n+1}-\mathbf{u}_h^{n}\|^2-\|\mathbf{u}_h^{n}\|^2-\|2\mathbf{u}_h^{n}-\mathbf{u}_h^{n-1}\|^2+\|\mathbf{u}_h^{n+1}-2\mathbf{u}_h^{n}+\mathbf{u}_h^{n-1}\|^2\right)\nonumber\\
& + \frac{2\Delta t^2}{3}(\|\nabla p_h^{n+1}\|^2-\|\nabla p_h^{n}\|^2) + \frac{3}{2}\|\mathbf{u}_h^{n+1}-\tilde{\mathbf{u}}_h^{n+1}\|^2 +|r^{n+1}|^2+|2r^{n+1}-r^{n}|^2-|r^{n}|^2 \nonumber\\
&-|2r^{n}-r^{n-1}|^2+|r^{n+1}-2r^{n}+r^{n-1}|^2=- \frac{2\Delta t}{Re} \|\sqrt{2\mu_{ph}^{*}}\mathbb{D}\tilde{\mathbf{u}}^{n+1}_h\|_F^2 \nonumber\\
& - 2\Delta t|\xi^{n+1}|^2\frac{Co}{Pe}\sum\limits_{i} \|\sqrt{ c_{ih}^{n+1}}\nabla g_{ih}^{n+1}\|^2.
\end{align}	
Finally, we can obtain \eqref{discrete-energy-law} from \eqref{discrete-energy-law1}.
\end{proof}

\subsection{Implementation}
It seems that the developed scheme is not a fully decoupled scheme. Any direct iterative methods needs much time consumptions. Therefore, we using split technique to eliminate all nonlocal terms and obtain a fully decoupled implementation.\\
\textbf{Step 1}. Find $\sigma_{ih}^{n+1}\in Q_h$ such that
\begin{align}
&(\frac{ 3\sigma_{ih}^{n+1} - 4\sigma_{ih}^{n} + \sigma_{ih}^{n-1} }{2\Delta t},\eta_{ih}) + (\mathbf{u}_h^{*}\cdot\nabla \sigma_{ih}^{n+1}, \eta_{ih})+ (\nabla \sigma_{ih}^{n+1},\nabla\eta_{ih})\nonumber \\
&= \frac{1}{Pe}\Big((\nabla \sigma_{ih}^{*}\cdot\nabla \sigma_{ih}^{n+1}, \eta_{ih}) - z_i(\nabla V_h^{*},\nabla\eta_i) + (z_i\nabla\sigma_{ih}^{n+1}\cdot\nabla V_h^{*}, \eta_{ih})\nonumber\\
&+ (\sum\limits_{j}\omega_{ij}\nabla\sigma_{ih}^{n+1}\cdot\nabla \sigma_{jh}^{*} c_{jh}^{*},\eta_{ih} ) -  (\omega_{ij}\nabla \sigma_{jh}^{*} c_{jh}^{*},\nabla\eta_{ih}) -  (\omega_{ii}\nabla \sigma_{ih}^{n+1} c_{ih}^{*},\nabla\eta_{ih}) \Big),\quad \forall \eta_i\in Q_h.
\end{align}
\textbf{Step 2}. Compute $c_{ih}^{n+1}$ by
\begin{align*}
&\bar{c}_{ih}^{n+1}=\exp(\sigma_{ih}^{n+1}),\\
&c_{ih}^{n+1}=\frac{( c_{ih}^{n},1) }{( \bar{c}_{ih}^{n+1},1)}\bar{c}_{ih}^{n+1}.
\end{align*}
\textbf{Step 3}. Find $\bar{V}_{h}^{n+1}\in S_h$ such that
\begin{align}\label{decoupled-numerical-Possion}
(\lambda\nabla \bar{V}_{h}^{n+1}, \nabla\phi_{h})= \sum\limits_{i}(z_ic_{ih}^{n+1}, \phi_{h}),\quad \forall \phi_h\in Q_h.
\end{align}
\textbf{Step 4}. Find $\tilde{\mathbf{u}}_{1h}^{n+1},\tilde{\mathbf{u}}_{2h}^{n+1}\in \bm{X}_h$ such that
\begin{align}
& (\frac{3\tilde{\mathbf{u}}_{1h}^{n+1} -4\mathbf{u}_{h}^{n} +\mathbf{u}_{h}^{n-1}}{2\Delta t},\mathbf{v}_{h} ) + \frac{1}{Re}(2\mu_{ph}^*\mathbb{D}\tilde{\mathbf{u}}_{1h}^{n+1},\mathbb{D}\mathbf{v}_{h})  = (\nabla\cdot\mathbf{v}_{h},  p_{h}^{n}),\quad \forall \mathbf{v}_h\in \bm{X}_h,\label{decoupled-numerical-NS1}\\
& (\frac{3\tilde{\mathbf{u}}_{2h}^{n+1}}{2\Delta t},\mathbf{v}_{h} ) + \frac{1}{Re}(2\mu_{ph}^*\mathbb{D}\tilde{\mathbf{u}}_{2h}^{n+1},\mathbb{D}\mathbf{v}_{h}) = - ((\mathbf{u}_{h}^{*}\cdot\nabla)\mathbf{u}_{h}^{*},\mathbf{v}_{h} ) - (\sum\limits_{i}Coz_ic_{ih}^{n+1}\nabla \bar{V}_{h}^{n+1},\mathbf{v}_{h} ),\quad \forall \mathbf{v}_h\in \bm{X}_h.\label{decoupled-numerical-NS2}
\end{align}
The above two subequations are obtained by splitting \eqref{discrete-numerical-NS}. We use the nonlocal scalar variable $\xi^{n+1}$ to split $\tilde{\mathbf{u}}_{h}^{n+1}$ into a linear combination that reads as
\begin{align}\label{split}
\tilde{\mathbf{u}}_{h}^{n+1}=\tilde{\mathbf{u}}_{1h}^{n+1} + \xi^{n+1}\tilde{\mathbf{u}}_{2h}^{n+1}.
\end{align}
By replacing $\tilde{\mathbf{u}}_{h}^{n+1}$ in \eqref{discrete-numerical-NS}, we have
\begin{align}\label{discrete-numerical-NSNS}
&\left(\frac{3(\tilde{\mathbf{u}}_{1h}^{n+1} + \xi^{n+1}\tilde{\mathbf{u}}_{2h}^{n+1}) -4\mathbf{u}^{n}_{h} +\mathbf{u}^{n-1}_{h}}{2\Delta t},\mathbf{v}_h\right) + \xi^{n+1}((\mathbf{u}^{*}_{h}\cdot\nabla)\mathbf{u}^{*}_{h},\mathbf{v}_h) + (\nabla p^{n}_{h},\mathbf{v}_h)  \nonumber\\
&+ \frac{1}{Re} \left(2\mu_{ph}^*\mathbb{D}(\tilde{\mathbf{u}}_{1h}^{n+1} + \xi^{n+1}\tilde{\mathbf{u}}_{2h}^{n+1}),\mathbb{D}\mathbf{v}_h\right) = -Co\xi^{n+1}(\sum\limits_{i}z_ic_{ih}^{n+1}\nabla \bar{V}_{h}^{n+1},\mathbf{v}_h).
\end{align}
According to $\xi^{n+1}$, \eqref{discrete-numerical-NSNS} can be decomposed into two subequations \eqref{decoupled-numerical-NS1} and \eqref{decoupled-numerical-NS2}.\\
\textbf{Step 5}. Compute $\xi^{n+1}$ by
\begin{align}\label{xixi}
\xi^{n+1} = \frac{ 2r^{n} -\frac{1}{2}r^{n-1} + \Delta t\zeta_1}{\frac{3}{2}\sqrt{\bar{E}_{SPNP}^{n+1}+ B} + \Delta t \zeta_2},
\end{align}
where
\begin{align*}
\zeta_1 &= \frac{1}{2\sqrt{\bar{E}_{SPNP}^{n+1} + B}}\left( Co (\sum\limits_{i}z_i c_{ih}^{n+1}\tilde{\mathbf{u}}^{n+1}_{1h} ,\nabla \bar{V}_h^{n+1} ) + ((\mathbf{u}_{h}^{*}\cdot\nabla)\mathbf{u}_{h}^{*}, \tilde{\mathbf{u}}^{n+1}_{1h} ) \right),\\
\zeta_2& = \frac{1}{2\sqrt{\bar{E}_{SPNP}^{n+1} + B}}\left( \frac{Co}{Pe}\sum\limits_{i}\|\sqrt{ c_{ih}^{n+1}}\nabla g_{ih}^{n+1}\|^2   -Co (\sum\limits_{i}z_i c_{ih}^{n+1}\tilde{\mathbf{u}}^{n+1}_{2h} ,\nabla \bar{V}_h^{n+1} ) - ((\mathbf{u}_{h}^{*}\cdot\nabla)\mathbf{u}_{h}^{*}, \tilde{\mathbf{u}}^{n+1}_{2h}  ) \right).
\end{align*}
Formula \eqref{xixi} is obtained by using the split form \eqref{split} in \eqref{discrete-numerical-r}. Equation \eqref{discrete-numerical-r} can be rewritten as 
\begin{align}
&\frac{3\xi^{n+1}\sqrt{\bar{E}_{SPNP}^{n+1}+ B} - 4r^{n} + r^{n-1}}{2\Delta t} = \frac{1}{2\sqrt{\bar{E}_{SPNP}^{n+1} + B}}\Big( -\frac{Co}{Pe}\xi^{n+1}\sum\limits_{i}\|\sqrt{ c_{ih}^{n+1}}\nabla g_{ih}^{n+1}\|^2 \nonumber\\
&+Co (\sum\limits_{i}z_i c_{ih}^{n+1}(\tilde{\mathbf{u}}_{1h}^{n+1} + \xi^{n+1}\tilde{\mathbf{u}}_{2h}^{n+1}) ,\nabla \bar{V}_h^{n+1} ) + ((\mathbf{u}_{h}^{*}\cdot\nabla)\mathbf{u}_{h}^{*}, (\tilde{\mathbf{u}}_{1h}^{n+1} + \xi^{n+1}\tilde{\mathbf{u}}_{2h}^{n+1})   ) \Big).
\end{align}
According to $\xi^{n+1}$, we can derive formula \eqref{xixi}. Then we verify that $\xi^{n+1}$ is solvable by showing $\sqrt{\bar{E}_{SPNP}^{n+1}+ B} + \Delta t \zeta_{2}\neq 0$. Taking $\mathbf{v}_h=\tilde{\bm{u}}_{2h}^{n+1}$ in \eqref{decoupled-numerical-NS2}, we have
\begin{align}
& - ((\mathbf{u}_{h}^{*}\cdot\nabla)\mathbf{u}_{h}^{*},\tilde{\bm{u}}_{2h}^{n+1} ) - Co(\sum\limits_{i}z_ic_{ih}^{n+1}\nabla \bar{V}_{h}^{n+1},\tilde{\bm{u}}_{2h}^{n+1} )= \frac{3}{2\Delta t}\|{\tilde{\bm{u}}}_{2h}^{n+1}\|^2 +\frac{ 1}{Re} \|\sqrt{2\mu_{ph}^{*}}\mathbb{D}\tilde{\mathbf{u}}^{n+1}_{2h}\|_F^2 \geq 0.
\end{align}
Therefore,
\begin{align*}
\zeta_{2} =& \frac{1}{2\sqrt{\bar{E}_{SPNP}^{n+1} + B}}\Big(\frac{Co}{Pe}\sum\limits_{i}\|\sqrt{ c_{ih}^{n+1}}\nabla g_{ih}^{n+1}\|^2 + \frac{3}{2\Delta t}\|{\tilde{\bm{u}}}_{2h}^{n+1}\|^2 +\frac{ 1}{Re} \|\sqrt{2\mu_{ph}^{*}}\mathbb{D}\tilde{\mathbf{u}}^{n+1}_{2h}\|_F^2 \Big)\geq 0,
\end{align*}
which implies $\sqrt{\bar{E}_{SPNP}^{n+1}+ B} + \Delta t \zeta_{2}\neq 0$.\\
\textbf{Step 6}. Update $r^{n+1}$, $V_{h}^{n+1}$ and $\tilde{\mathbf{u}}_{h}^{n+1}$ by
\begin{align*}
&r^{n+1} = \xi^{n+1}\sqrt{\bar{E}_{SPNP}^{n+1}+ B},\\
&V_{h}^{n+1} = \xi^{n+1}\bar{V}_{h}^{n+1},\\
&\tilde{\mathbf{u}}_{h}^{n+1} = \tilde{\mathbf{u}}_{1h}^{n+1}+ \xi^{n+1}\tilde{\mathbf{u}}_{1h}^{n+1}.
\end{align*}
\textbf{Step 7}: Find $\psi^{n+1}\in M_h$ such that
\begin{align*}
(\nabla \psi^{n+1}_h,\nabla q_h) = \frac{3}{2\Delta t}(\tilde{\mathbf{u}}_{h}^{n+1},\nabla q_h),\quad \forall \psi_h\in M_h .
\end{align*}
\textbf{Step 8}: Update $\mathbf{u}_h^{n+1}$, $p_h^{n+1}$, $\mu_{ph}^{n+1}$ by
\begin{align*} 
&\mathbf{u}_h^{n+1} = \tilde{\mathbf{u}}_h^{n+1} - \frac{2\Delta t}{3}\nabla \psi_h^{n+1},\\
&p^{n+1}_h=  \psi^{n+1}_h + p_h^{n},\\
& \mu_{ph}^{n+1} = \mu_{\infty} +(\mu_0 - \mu_{\infty})\Big( 1 + \lambda_1^2(2\mathbb{D}\mathbf{u}^{n+1}_{h}:\mathbb{D}\mathbf{u}^{n+1}_{h})\Big)^{\frac{k-1}{2}}.
\end{align*}

\section{Numerical results}

In this section, we first implement several numerical examples to verify the convergence and energy stability of the proposed scheme \eqref{numerical-model}. Then, some benchmark simulations are proposed to demonstrate the effectiveness of the scheme. 

In all numerical experiments below, we set the computational domain to be a rectangular domain $\Omega=[0,1]^2$ and use the Taylor-Hood element $(\bm{P}_2, P_1)$ for velocity and pressure. The $P_2$ element is also used for electric potential and ion concentrations. And we consider only two types of ions: positive ion (denoted by $i = p $) and negative ion (denoted by $ i = n $). Correspondingly, their valence numbers \( z_i \) are set to be \( z_p=1 \) and \( z_n=-1 \), respectively. Unless otherwise specified, all boundary conditions follow the default settings \eqref{BC} established in this work. The computations are performed with software Freefem $++$~\cite{hecht2012new}.

\subsection{Accuracy test}
The computational domain is assumed to be a rectangular region $\Omega=[0,1]^2$, and the terminal time is set to be $T=0.5$. We take the steric interaction coefficient matrix
$\left(
  \begin{array}{cc}
    2 & 1 \\
    1 & 2 \\
  \end{array}
\right)$.
The parameters in the model are set as $\lambda=1$, $Pe=2$, $Re=1$, $Co=5$, $k=0.5$, $\mu_0=1.0$, $\mu_\infty=0.5$, $\lambda_1=1$. Assuming the exact solutions of the system \eqref{nondimensional-model} are given as functions
\begin{equation}
\begin{split}
&c_p(x,y,t)=1.2 + \cos(\pi x)\cos(\pi y)\exp(-t), \\
&c_n(x,y,t)=1.2 - \cos(\pi x)\cos(\pi y)\exp(-t), \\
&V(x,y,t)=\frac{1}{\pi^2}\cos(\pi x)\cos(\pi y)\exp(-t),\\
&\bm{u}(x,y,t)=(\pi\sin^2(\pi x)\sin(2\pi y)\exp(-t),-\pi\sin(2\pi x)\sin^2(\pi y)\exp(-t)),\\
&p(x,y,t)=\cos(\pi x)\cos(\pi y)\exp(-t).
\end{split}
\end{equation}
Some source terms are added so that the above solutions can satisfy the \eqref{nondimensional-model}.
The mesh size in space is set to be $\sqrt{2}/256$ so that it is negligible compared to the time discretization error.
We use quadratic element for $\bm{u}, c_i, V$ and linear element for $p$.
The errors in $L^2$ norm between the numerical solutions and the exact solutions are shown
in Table \ref{tab:L2accuracy}.
It is clear to see that the order of convergence in time is approximately to be second-order.

\begin{table}[tbhp]
{\small
   \caption{Temporal convergence for the velocity $\bm{u}$, pressure $p$, concentration $c_p$, $c_n$ and electric potential $V$ by using the $L^2$ norm. \label{tab:L2accuracy}}
  \begin{center}
     \begin{tabular}{llllllllll}
    \hline
       $N$      & $\|e_{\bm{u}}\|_2$   &order   &$\|e_{p}\|_2$    &order   &$\|e_{c_p}\|_2$   &order    \\
    \hline
    $16$         & 4.5707e-04         &-    & 1.2785e-02           & -  &   8.8518e-05           &-
    \\
    $32$ &1.1072e-04       &2.05   & 3.1288e-03            &  2.03  &  1.4312e-05           &2.63
    \\
    $64$ & 2.7404e-05      &2.01   & 7.7970e-04           &  2.00  &  2.2599e-06              &2.67
    \\
    $128$ & 6.8274e-06      &2.00   &  1.9952e-04           &  1.97 &  4.5885e-07             &  2.30
    \\
     \hline
     $N$       &$\|e_{c_n}\|_2$     &order  &$\|e_{V}\|_2$    &order    \\
    \hline
    $16$
    &  5.2692e-05       & -&7.0894e-06        & - \\
    $32$
    & 1.0068e-05       &  2.39& 1.2183e-06        & 2.54     \\
    $64$
    &  1.9913e-06       & 2.34& 2.1130e-07       & 2.53 \\
    $128$
    &   4.5245e-07       & 2.14&4.5142e-08        & 2.23     \\
    \hline
     \end{tabular}
  \end{center}
}
\end{table}

\subsection{The energy dissipation and mass conservation}
To show that the numerical scheme satisfies the discrete energy law and mass conservation, the Coulomb-driven flow in a cavity is considered. The terminal time is set to be $T=2$. We take the steric interaction coefficient matrix
$W=\left(
  \begin{array}{cc}
    2 & 0 \\
    0 & 2 \\
  \end{array}\right)$.
The parameters are set as $\lambda=0.2$, $Pe=50$, $Re=1$, $Co=0.6$, $k=0.2$, $\mu_0=1.5$, $\mu_\infty=0.5$, $\lambda_1=0.1$. The spatial mesh size is chosen to be $h=\frac{\sqrt{2}}{40}$.
The initial conditions are taken to be
\begin{equation}\label{eq:EnergydecayInitial}
\begin{split}
&c_{p0}=12 + 10\cos(\pi x)\cos(\pi y), \\
&c_{n0}=12 - 10\cos(\pi x)\cos(\pi y), \\
&\bm{u}_{0}=(0,0),
\end{split}
\end{equation}
where the initial condition for the electric potential satisfies $\int_{K_h} V d\bm{x} =0$. In Fig.~\ref{TimeDecaylaww}, we present the time evolution of the discrete total energy and variable $\xi$ for different time step $\Delta t$, one can see that the numerical scheme dissipates the total energy which is consistent with the theoretical analysis.
Fig.~\ref{TimeDecaylaw} depicts the values of the auxiliary variable $\xi$ with diverse time step sizes. One can see that the numerical solutions of the variable is approximately equal to 1. Fig.~\ref{MMassConservation} plots time evolution of the discrete masses $\int_{\Omega} c_p d\mathbf{x}$ and $\int_{\Omega} c_n d\mathbf{x}$, which demonstrates the mass conservation property of the scheme. 

\begin{figure}[htb]
\setlength{\abovecaptionskip}{0pt}
\setlength{\belowcaptionskip}{3pt}
\renewcommand*{\figurename}{Fig.}
\centering
\begin{minipage}[t]{0.41\linewidth}
\includegraphics[width=2.5in, height=2.2 in]{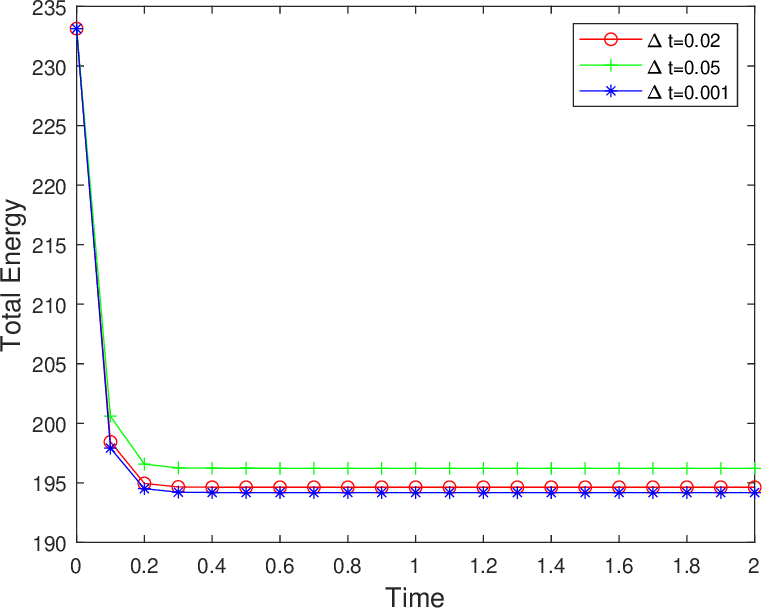}
\end{minipage}%
\caption{ Time evolution of the discrete total energy .\label{TimeDecaylaww}}
\end{figure}
\begin{figure}[htb]
\setlength{\abovecaptionskip}{0pt}
\setlength{\belowcaptionskip}{3pt}
\renewcommand*{\figurename}{Fig.}
\centering
\begin{minipage}[t]{0.41\linewidth}
\includegraphics[width=2.5in, height=2.2 in]{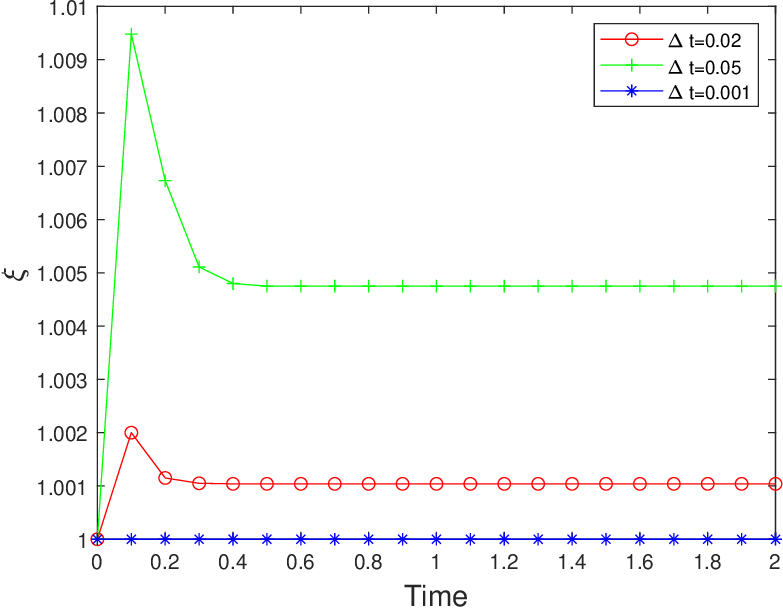}
\end{minipage}%
\caption{ Time evolution of the auxiliary variable $\xi$.\label{TimeDecaylaw}}
\end{figure}

\begin{figure}[htbp]
\setlength{\abovecaptionskip}{1pt}
\renewcommand*{\figurename}{Fig.}
\centering
\begin{minipage}[t]{0.41\linewidth}
\includegraphics[width=2.15in, height=1.8in]{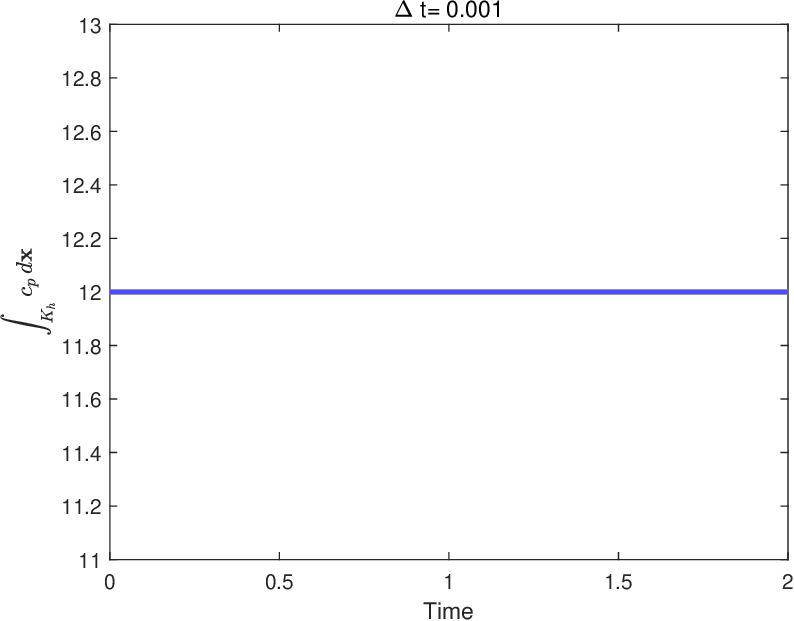}
\end{minipage}
\begin{minipage}[t]{0.41\linewidth}
\includegraphics[width=2.15in, height=1.8in]{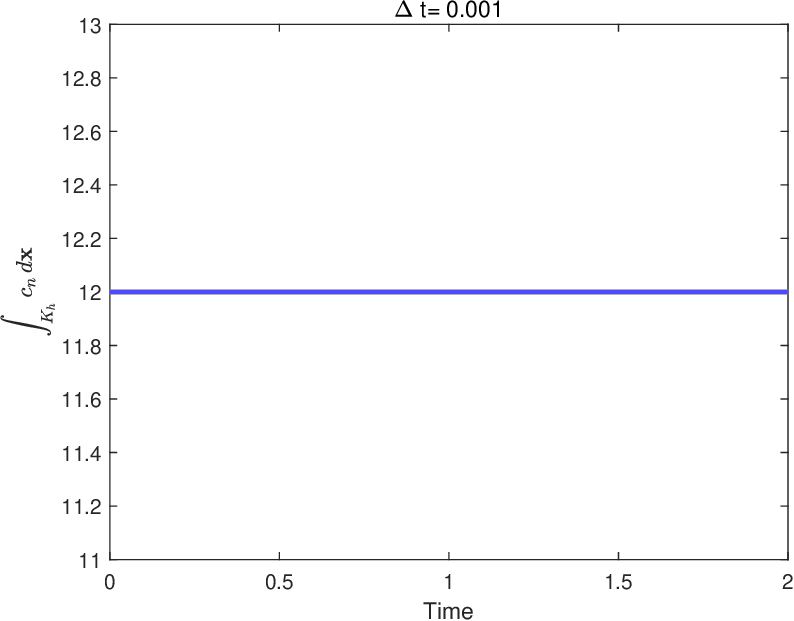}
\end{minipage}
\caption{Time evolution of the discrete masses $\int_{\Omega} c_p d\mathbf{x}$ and $\int_{\Omega} c_n d\mathbf{x}$.\label{MMassConservation}}
\end{figure}

\subsection{Steric effect}
To study the effect of steric interactions on the ion distribution, we take the steric interaction coefficient matrices  
$W = \left(
                        \begin{array}{cc}
                          0 & 0 \\
                          0 & 0 \\
                        \end{array}
                      \right)
$ , $\left(
                        \begin{array}{cc}
                          4 & 1 \\
                          1 & 4 \\
                        \end{array}
                      \right)$, $\left(
                        \begin{array}{cc}
                          8 & 1 \\
                          1 & 8 \\
                        \end{array}
                      \right)$, $\left(
                        \begin{array}{cc}
                          8 & 4 \\
                          4 & 8 \\
                        \end{array}
                      \right)$, $\left(
                        \begin{array}{cc}
                          8 & 7 \\
                          7 & 8 \\
                        \end{array}
                      \right)$. 
The initial conditions are given as follows
$$
\begin{aligned}
& \boldsymbol{u}_0=(0,0), \\
& c_{p 0}=10^{-6}+\left(1-10^{-6}\right) \cdot H(x) \cdot H_1(y), \\
& c_{n 0}=10^{-6}+\left(1-10^{-6}\right) \cdot H(x) \cdot H_2(y),
\end{aligned}
$$
where
$$
\begin{aligned}
& H(x)=\frac{1}{2}\left[1+\tanh \left(\frac{x-0.75}{0.04}\right)\right], \\
& H_1(y)=\frac{1}{2}\left[1+\tanh \left(\frac{y-0.55}{0.04}\right)\right], \\
& H_2(y)=\frac{1}{2}\left[1+\tanh \left(\frac{0.45-y}{0.04}\right)\right].
\end{aligned}
$$
We set the computed domain to be $\Omega=[0,1]^2$. In the computations, the parameters are set as $\lambda=0.1$, $Pe=50$, $Re=5$, $Co=5$, $k=0.5$, $\mu_0=1.0$, $\mu_{\infty}=0.5$, $\lambda_1=1$, $\Delta t=0.00$1, $h=\frac{\sqrt{2}}{40}$. The snapshots of $c_p$, $c_n$ and $V$ at $t = 0.002$, $t = 0.1$ and $t = 1$ for various $W$ are shown in Fig. \ref{StericEx1} - Fig. \ref{StericEx4}, respectively. With increasing diagonal entries in $W$, enhanced self steric interactions among same species of ions drives spatial dispersion, resulting in both broader distribution and reduced peak concentration as ions maximize the distances between same species of ions. As the off-diagonal entries in $W$ increases, stronger cross steric interactions between different species of ions cause the peak concentration to increase and the non-zero concentration region to shrink.

\begin{figure}[htb]
\setlength{\abovecaptionskip}{0pt}
\setlength{\abovecaptionskip}{0pt}
\setlength{\belowcaptionskip}{3pt}
\renewcommand*{\figurename}{Fig.}
\centering
\begin{minipage}[t]{0.29\linewidth}
\includegraphics[width=1.8in, height=1.6 in]{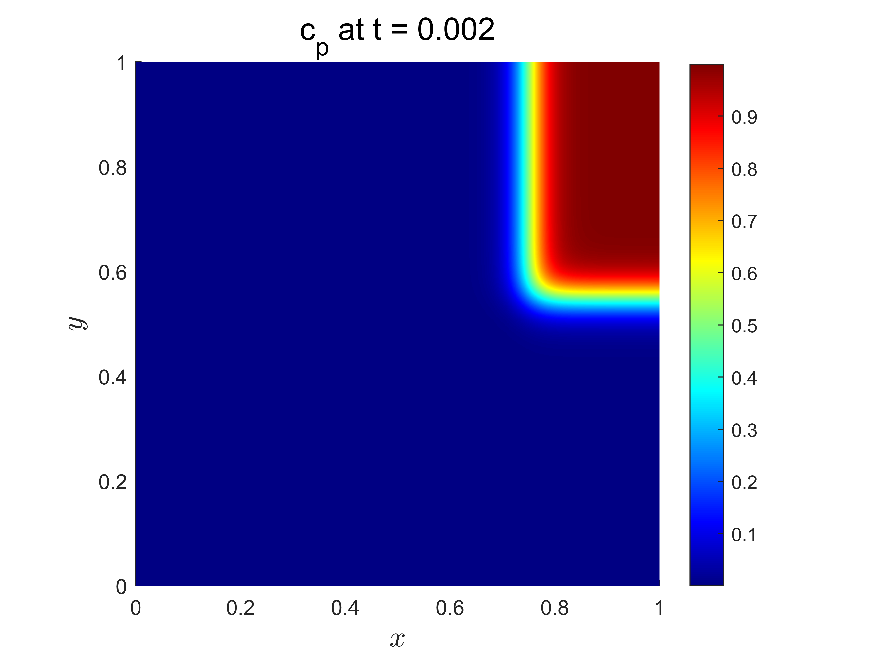}
\end{minipage}
\begin{minipage}[t]{0.29\linewidth}
\includegraphics[width=1.8in, height=1.6 in]{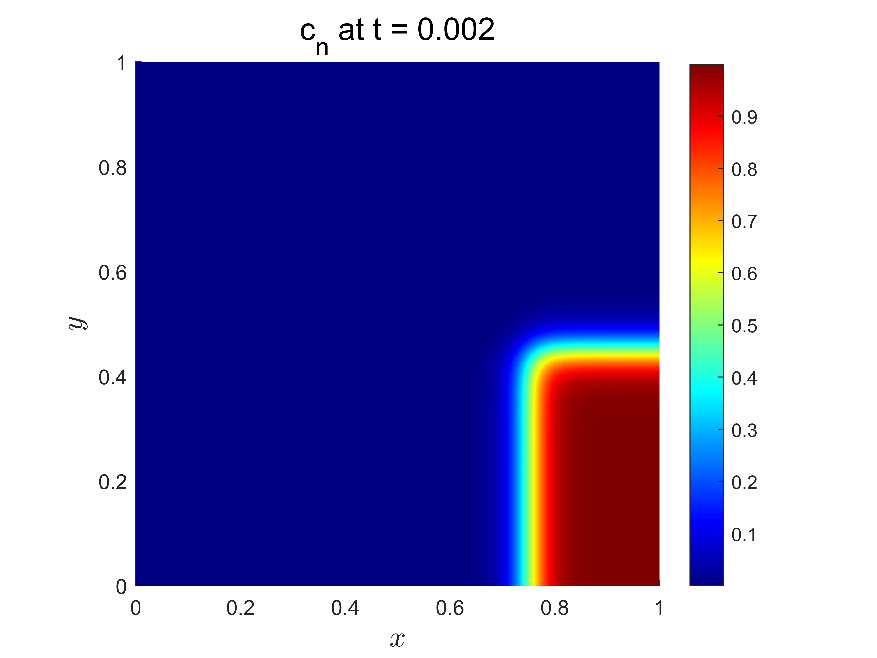}
\end{minipage}
\begin{minipage}[t]{0.29\linewidth}
\includegraphics[width=1.8in, height=1.6 in]{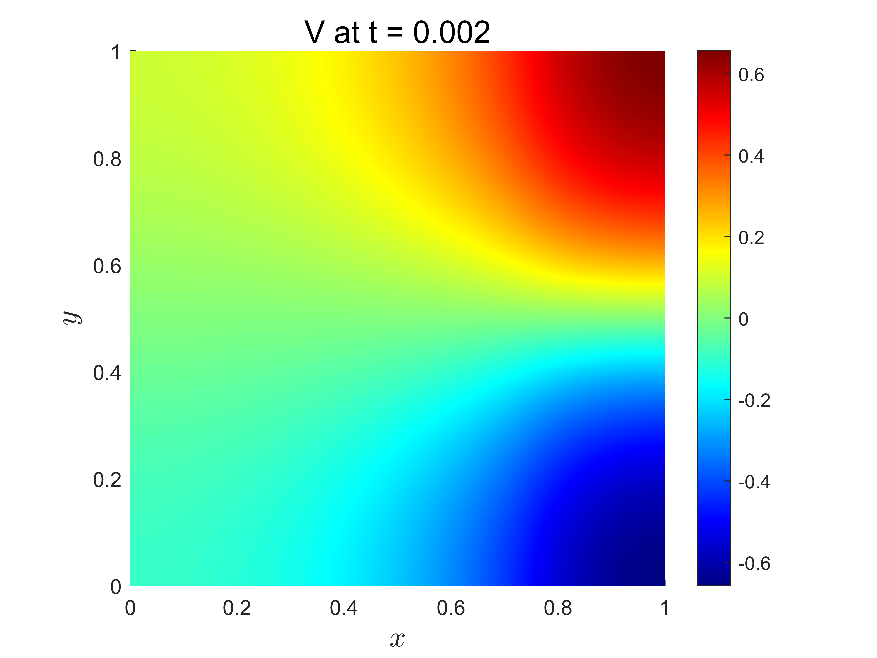}
\end{minipage}
\\
\begin{minipage}[t]{0.29\linewidth}
\includegraphics[width=1.8in, height=1.6in]{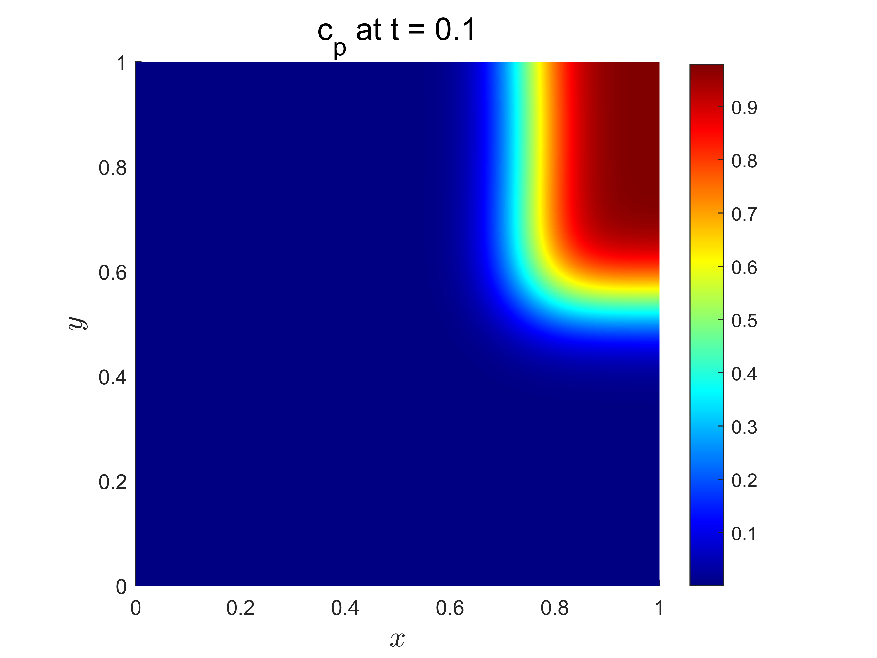}
\end{minipage}
\begin{minipage}[t]{0.29\linewidth}
\includegraphics[width=1.8in, height=1.6in]{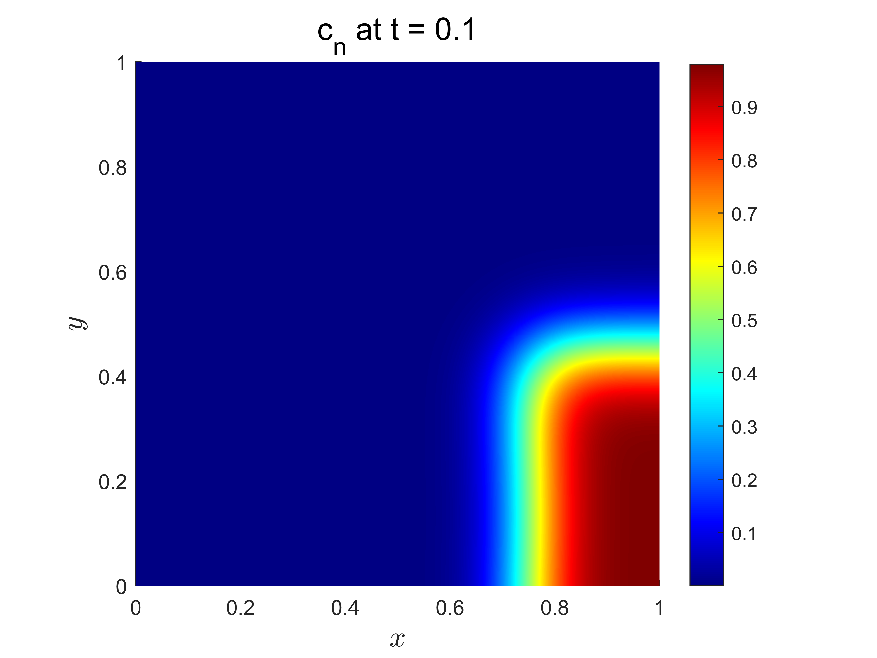}
\end{minipage}
\begin{minipage}[t]{0.29\linewidth}
\includegraphics[width=1.8in, height=1.6in]{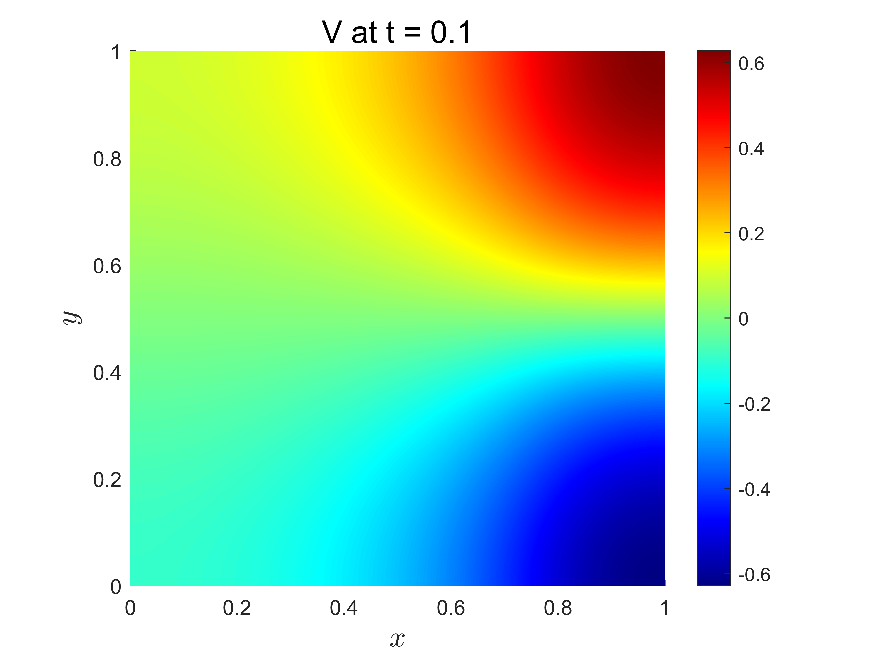}
\end{minipage}
\\
\begin{minipage}[t]{0.29\linewidth}
\includegraphics[width=1.8in, height=1.6in]{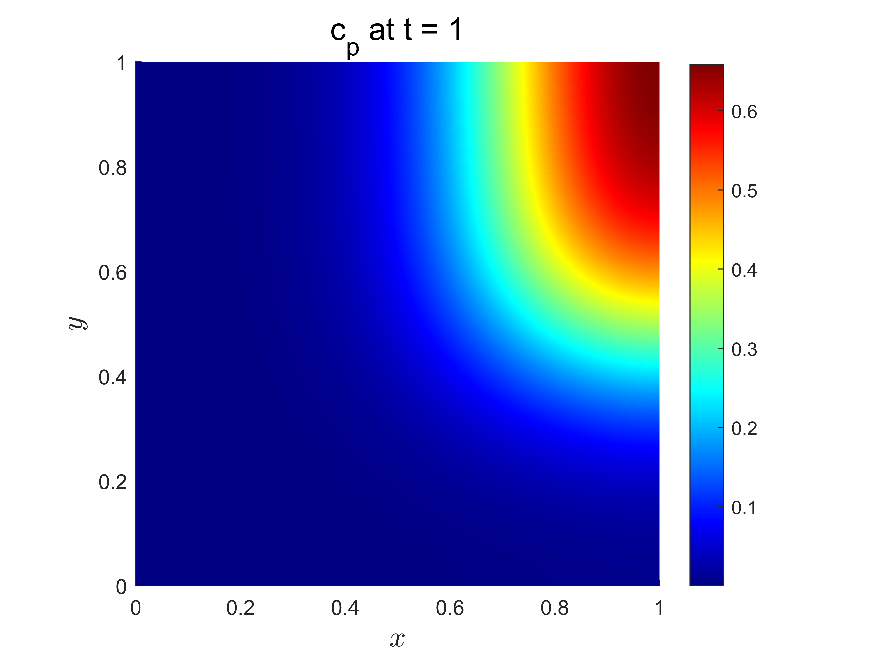}
\end{minipage}
\begin{minipage}[t]{0.29\linewidth}
\includegraphics[width=1.8in, height=1.6 in]{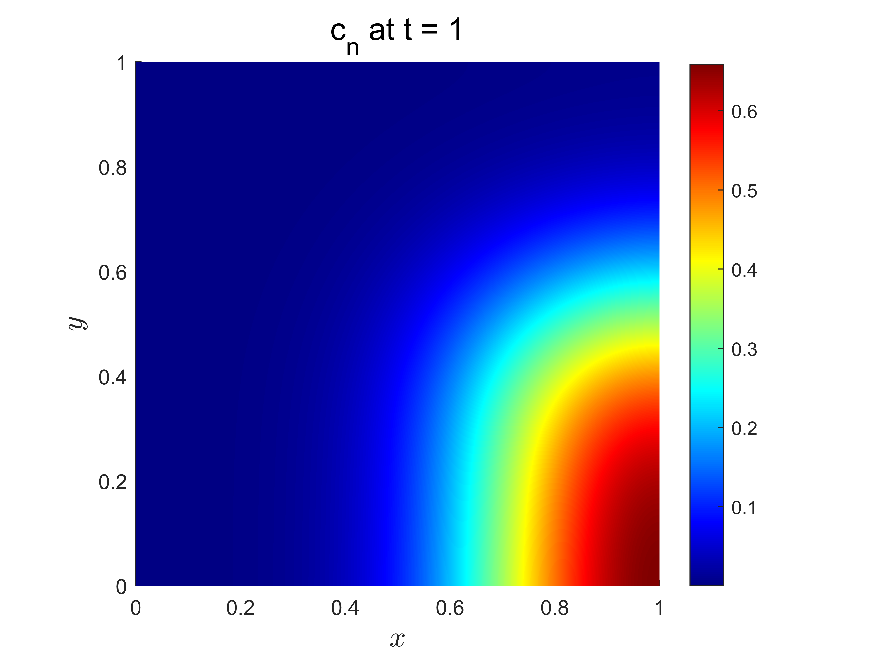}
\end{minipage}
\begin{minipage}[t]{0.29\linewidth}
\includegraphics[width=1.8in, height=1.6in]{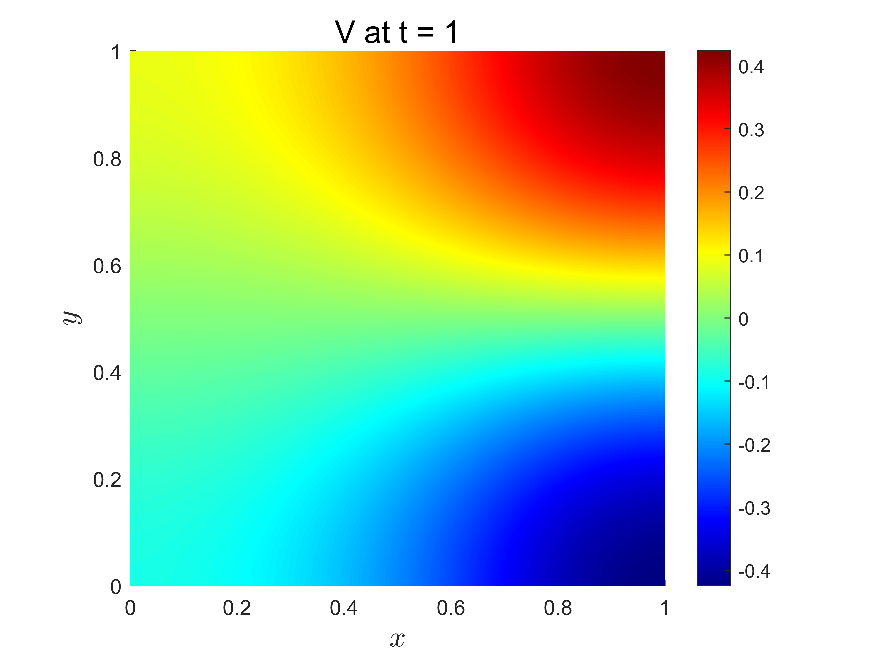}
\end{minipage}
\caption{Snapshots of $c_p$, $c_n$ and $V$ for $\omega_{11}=\omega_{22}=0$, $\omega_{12}=\omega_{21}=0$ at $t = 0.002$, $t = 0.1$ and $t = 1$.\label{StericEx1}}
\end{figure}

\begin{figure}[htb]
\setlength{\abovecaptionskip}{0pt}
\setlength{\abovecaptionskip}{0pt}
\setlength{\belowcaptionskip}{3pt}
\renewcommand*{\figurename}{Fig.}
\centering
\begin{minipage}[t]{0.29\linewidth}
\includegraphics[width=1.8in, height=1.6 in]{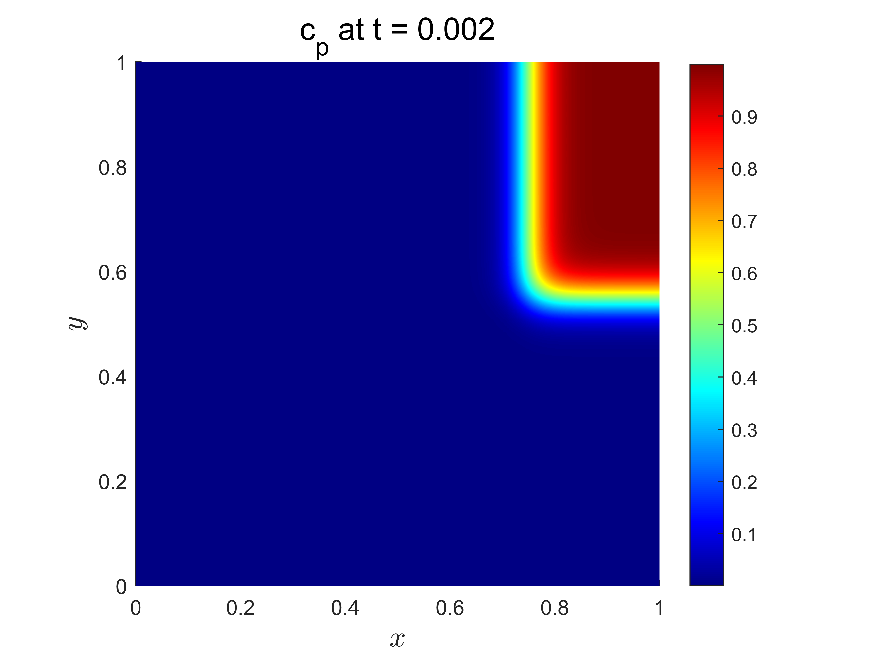}
\end{minipage}
\begin{minipage}[t]{0.29\linewidth}
\includegraphics[width=1.8in, height=1.6 in]{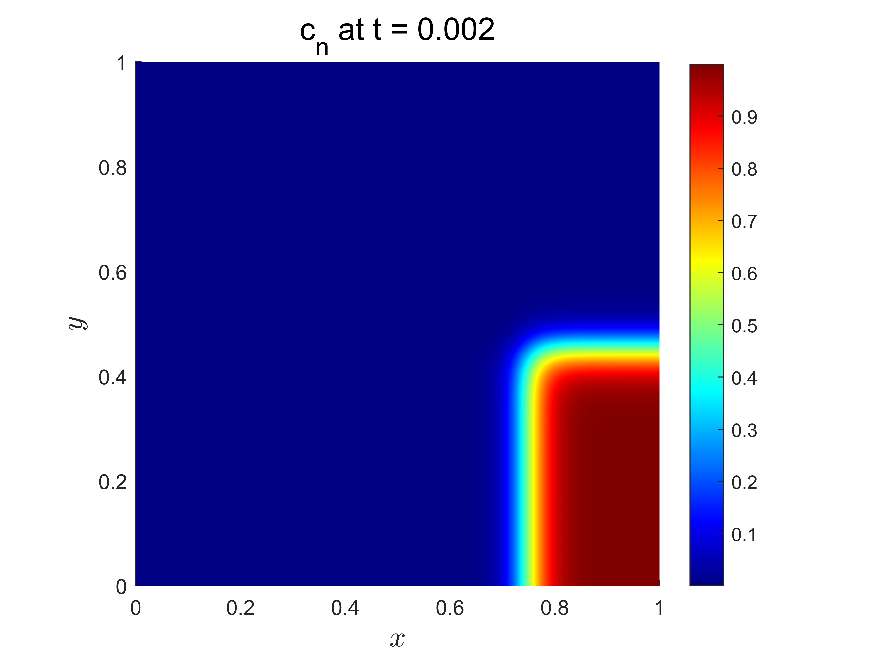}
\end{minipage}
\begin{minipage}[t]{0.29\linewidth}
\includegraphics[width=1.8in, height=1.6 in]{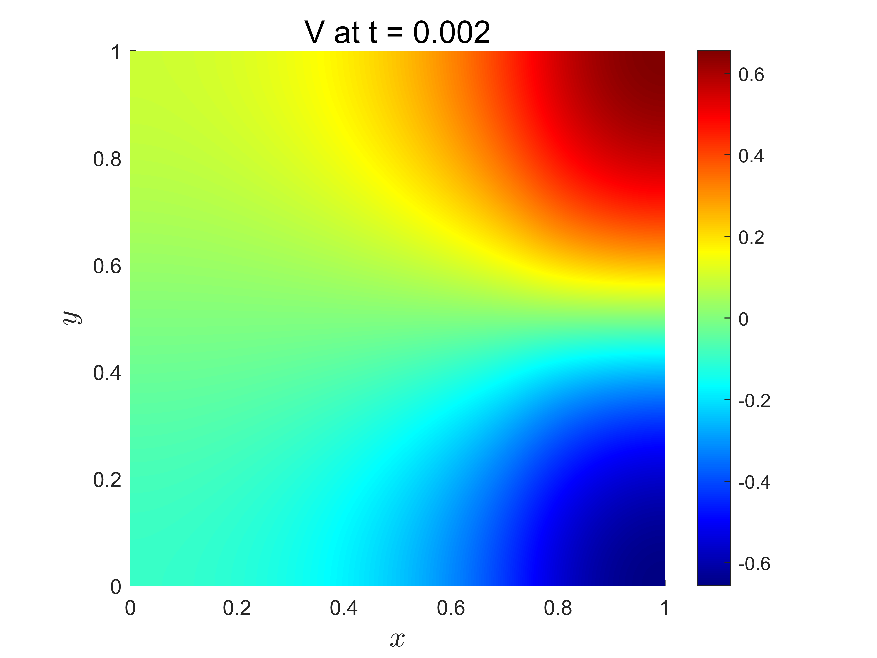}
\end{minipage}
\\
\begin{minipage}[t]{0.29\linewidth}
\includegraphics[width=1.8in, height=1.6 in]{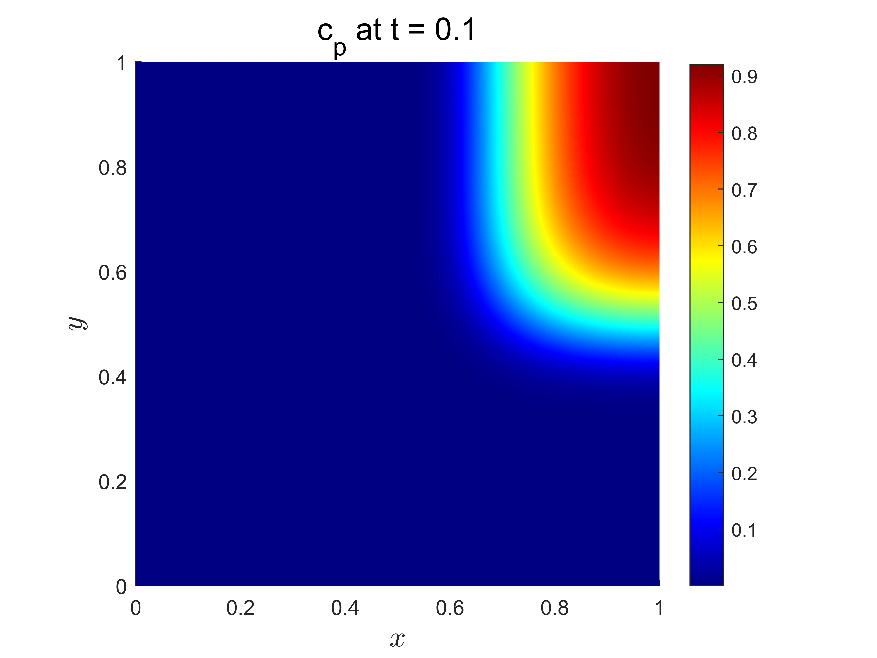}
\end{minipage}
\begin{minipage}[t]{0.29\linewidth}
\includegraphics[width=1.8in, height=1.6 in]{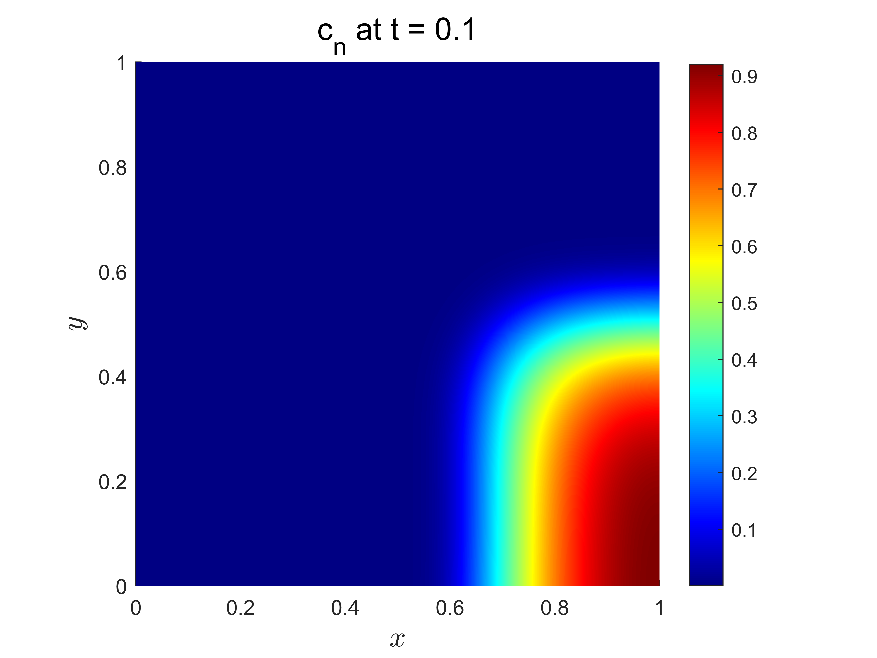}
\end{minipage}
\begin{minipage}[t]{0.29\linewidth}
\includegraphics[width=1.8in, height=1.6 in]{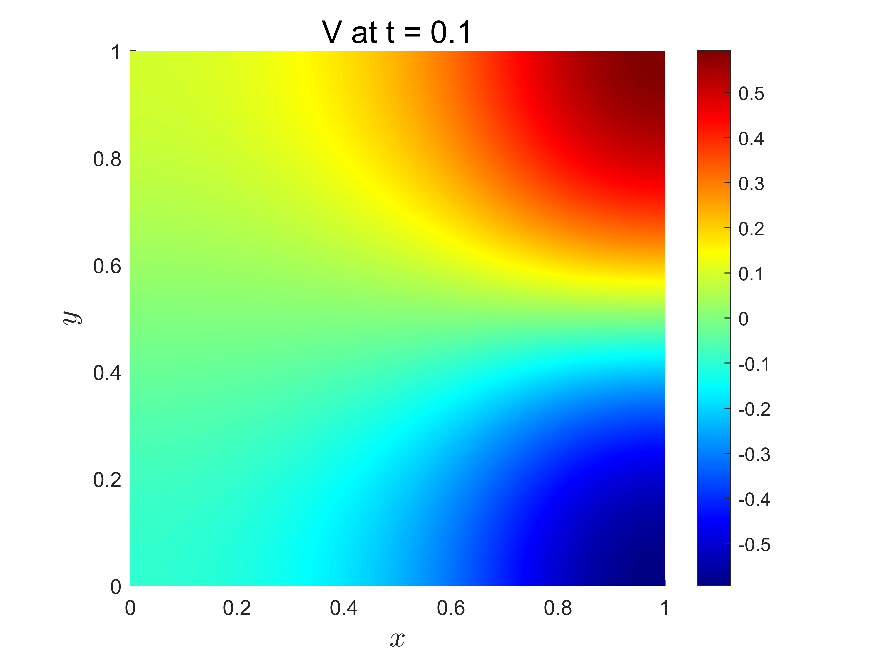}
\end{minipage}
\\
\begin{minipage}[t]{0.29\linewidth}
\includegraphics[width=1.8in, height=1.6 in]{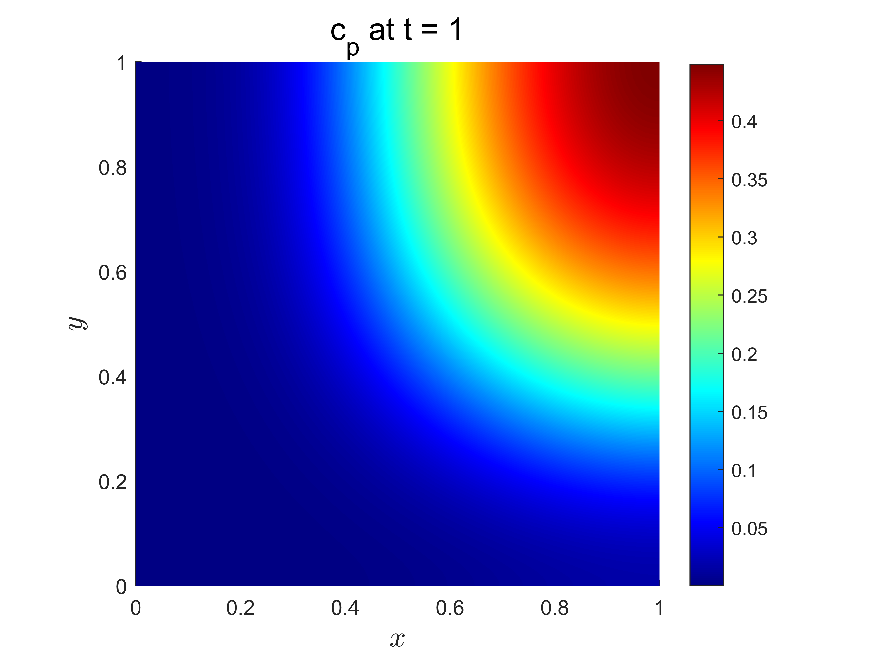}
\end{minipage}
\begin{minipage}[t]{0.29\linewidth}
\includegraphics[width=1.8in, height=1.6 in]{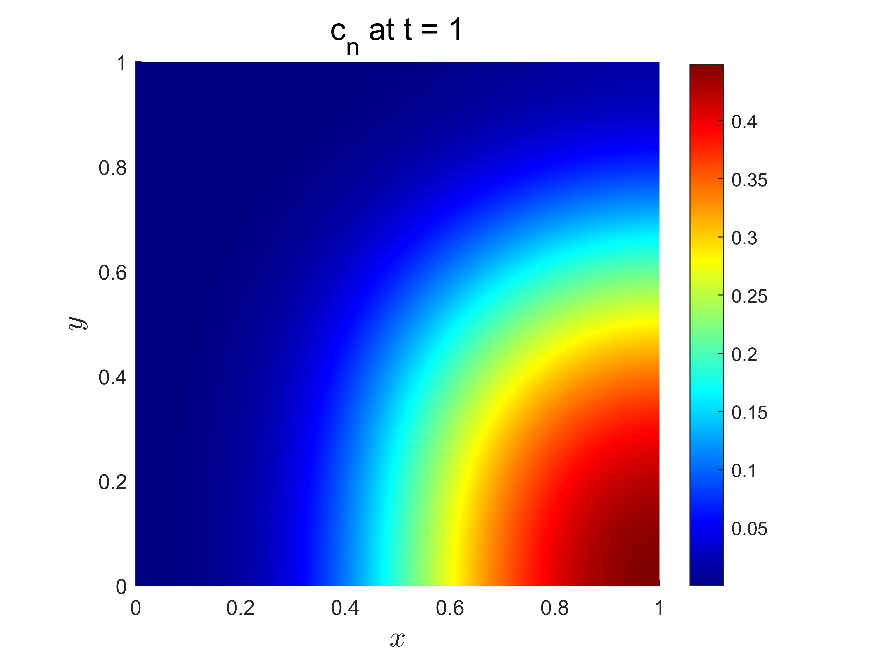}
\end{minipage}
\begin{minipage}[t]{0.29\linewidth}
\includegraphics[width=1.8in, height=1.6 in]{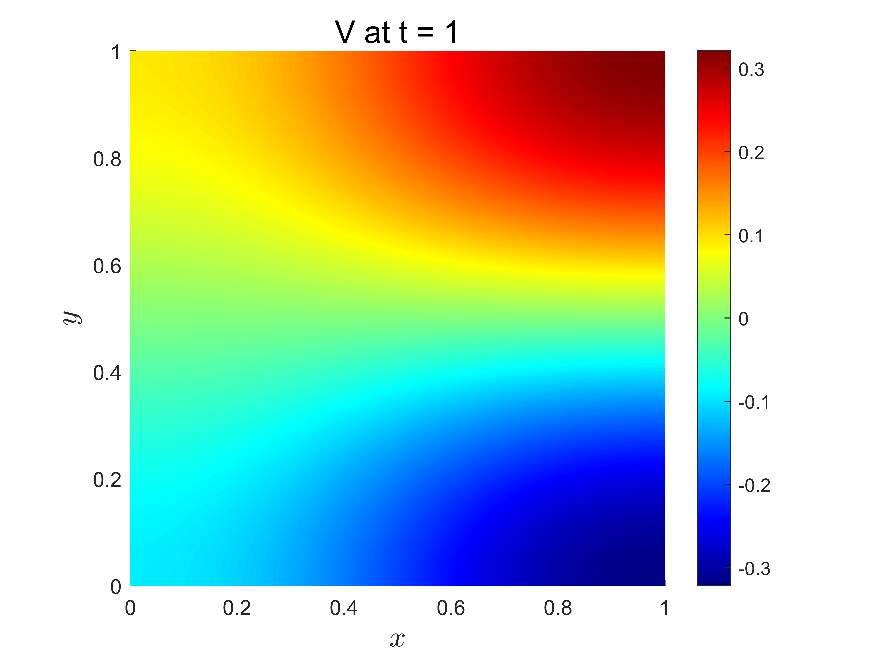}
\end{minipage}
\caption{Snapshots of $c_p$, $c_n$ and $V$ for $\omega_{11}=\omega_{22}=4$, $\omega_{12}=\omega_{21}=1$ at $t = 0.002$, $t = 0.1$ and $t = 1$.\label{StericEx2}}
\end{figure}

\begin{figure}[htb]
\setlength{\abovecaptionskip}{0pt}
\setlength{\abovecaptionskip}{0pt}
\setlength{\belowcaptionskip}{3pt}
\renewcommand*{\figurename}{Fig.}
\centering
\begin{minipage}[t]{0.29\linewidth}
\includegraphics[width=1.8in, height=1.6 in]{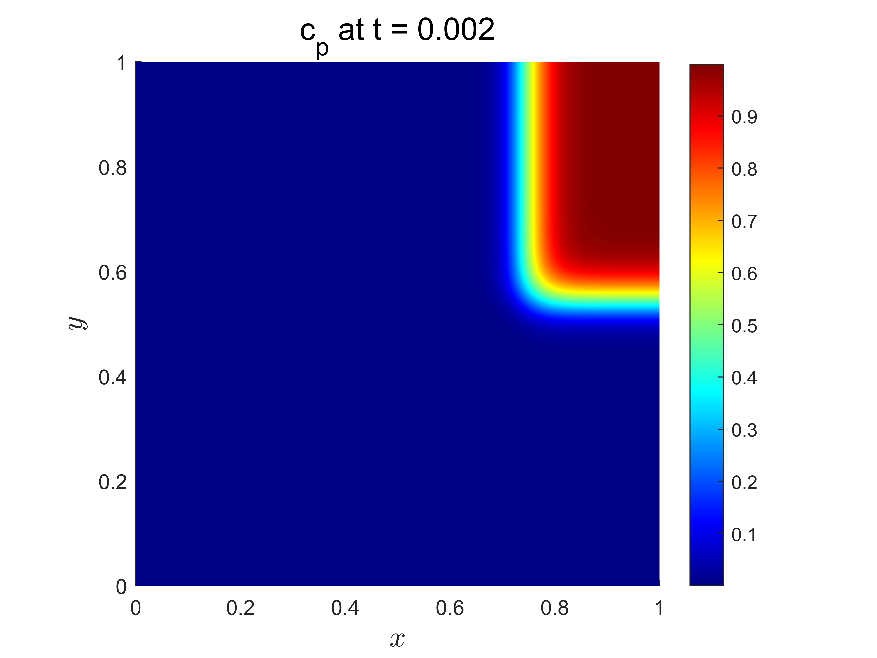}
\end{minipage}
\begin{minipage}[t]{0.29\linewidth}
\includegraphics[width=1.8in, height=1.6 in]{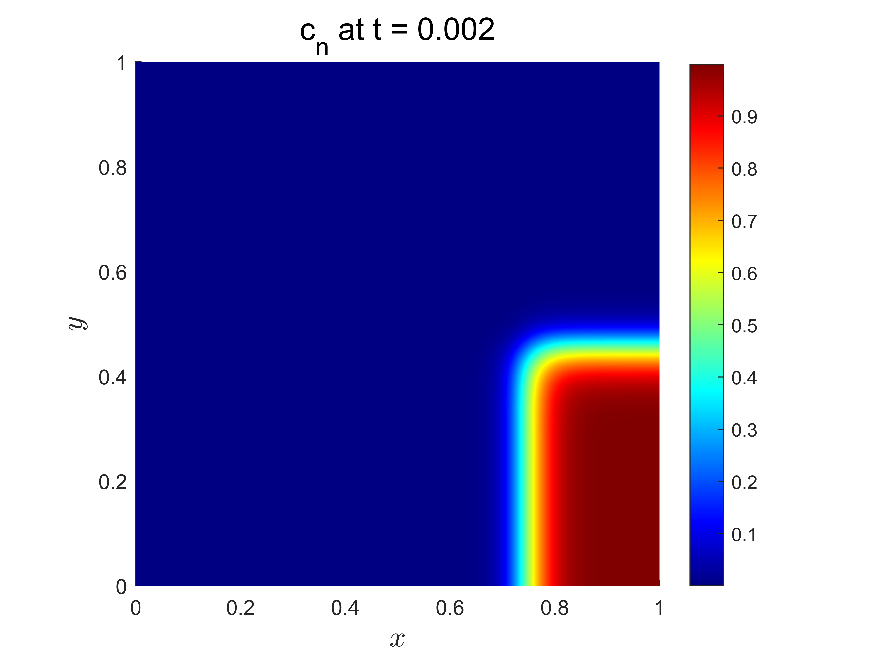}
\end{minipage}
\begin{minipage}[t]{0.29\linewidth}
\includegraphics[width=1.8in, height=1.6 in]{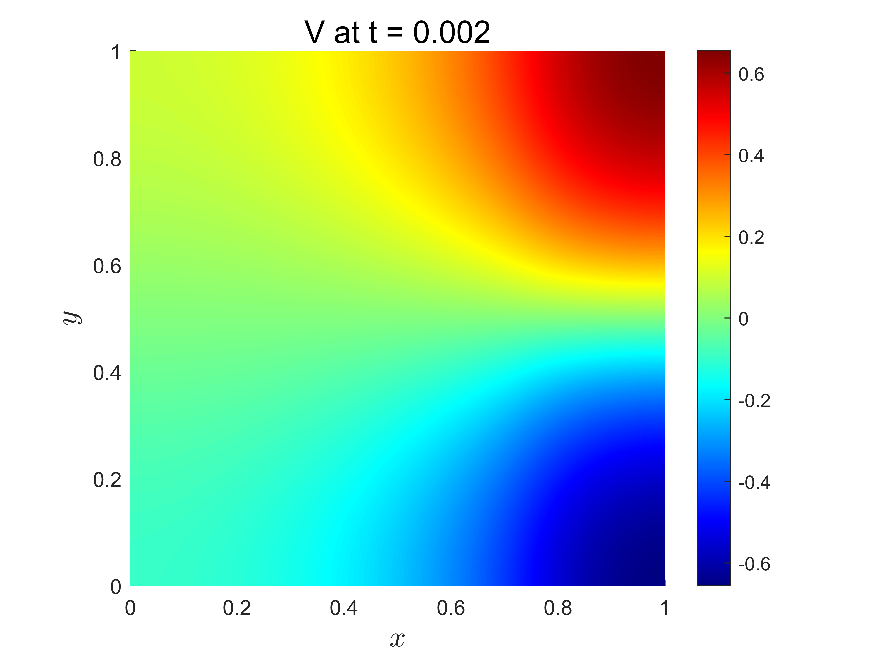}
\end{minipage}
\\
\begin{minipage}[t]{0.29\linewidth}
\includegraphics[width=1.8in, height=1.6 in]{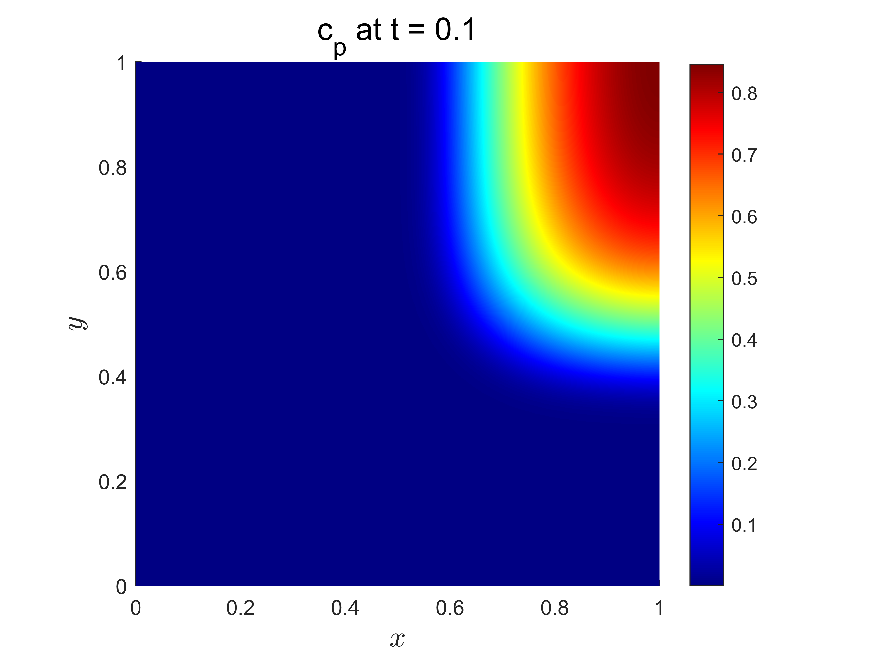}
\end{minipage}
\begin{minipage}[t]{0.29\linewidth}
\includegraphics[width=1.8in, height=1.6 in]{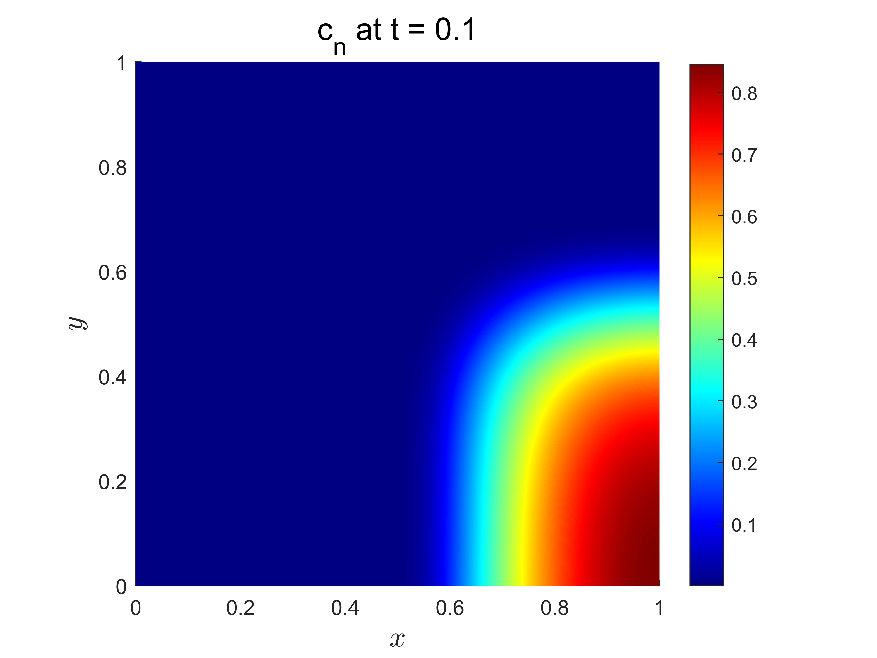}
\end{minipage}
\begin{minipage}[t]{0.29\linewidth}
\includegraphics[width=1.8in, height=1.6 in]{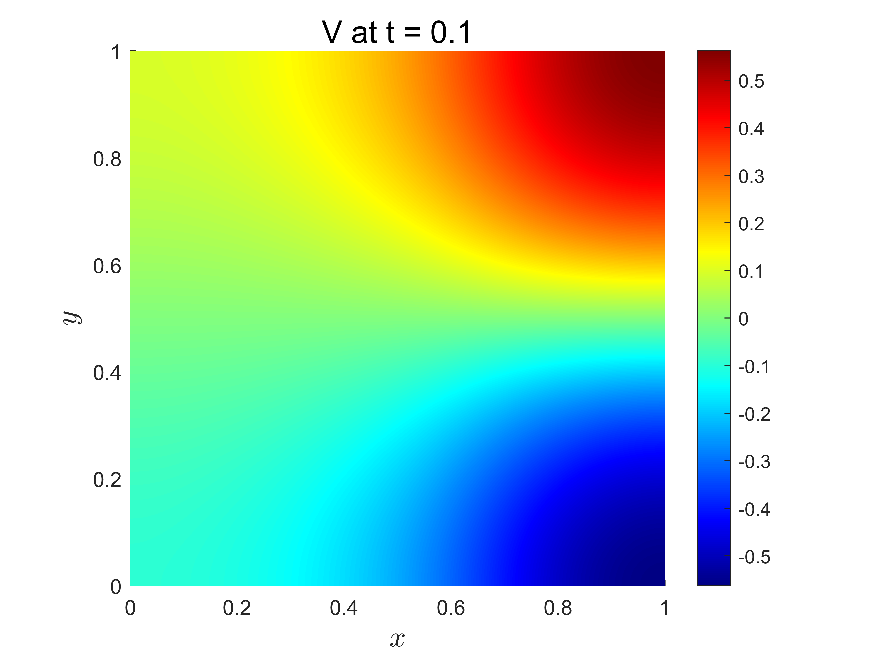}
\end{minipage}
\\
\begin{minipage}[t]{0.29\linewidth}
\includegraphics[width=1.8in, height=1.6 in]{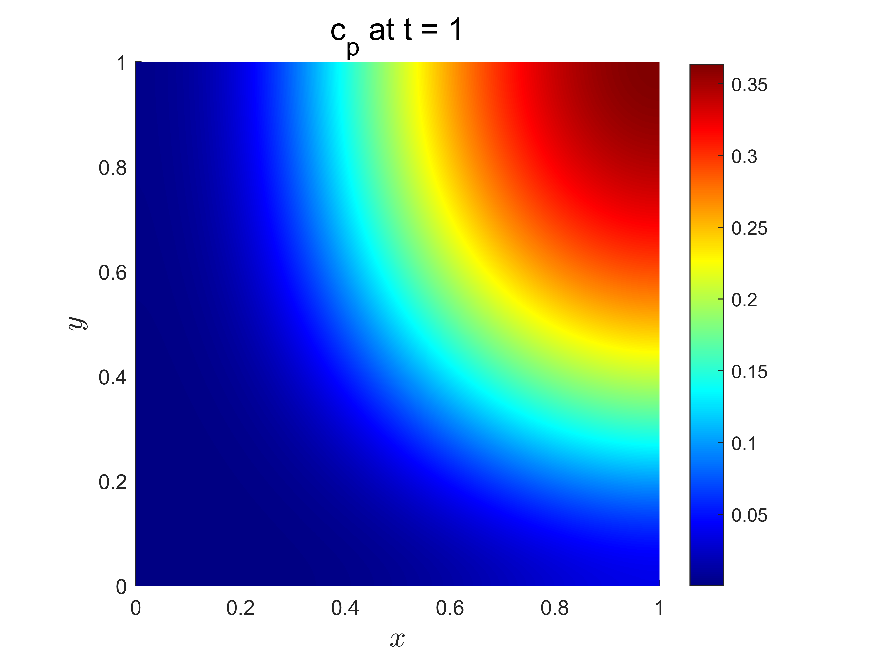}
\end{minipage}
\begin{minipage}[t]{0.29\linewidth}
\includegraphics[width=1.8in, height=1.6 in]{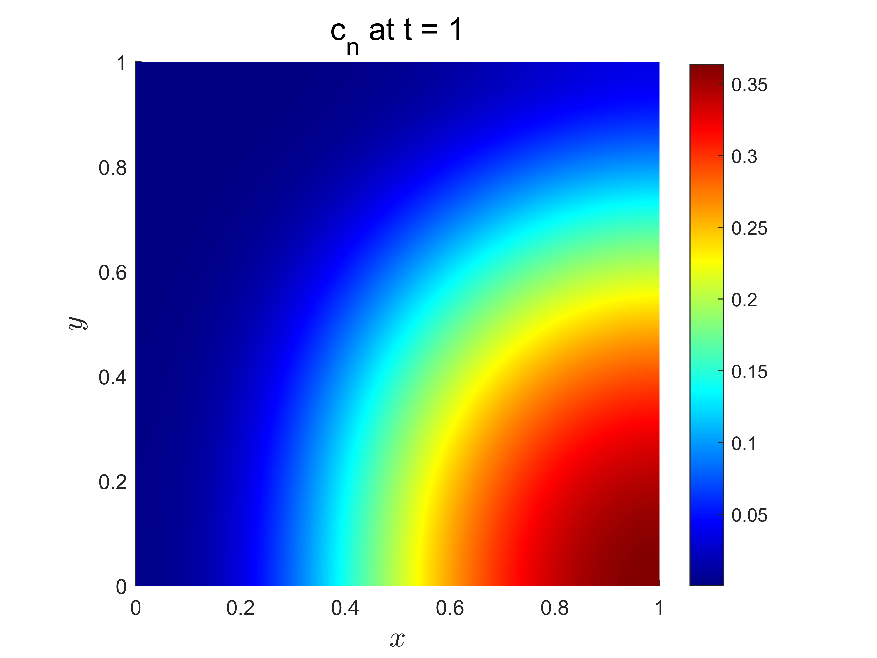}
\end{minipage}
\begin{minipage}[t]{0.29\linewidth}
\includegraphics[width=1.8in, height=1.6 in]{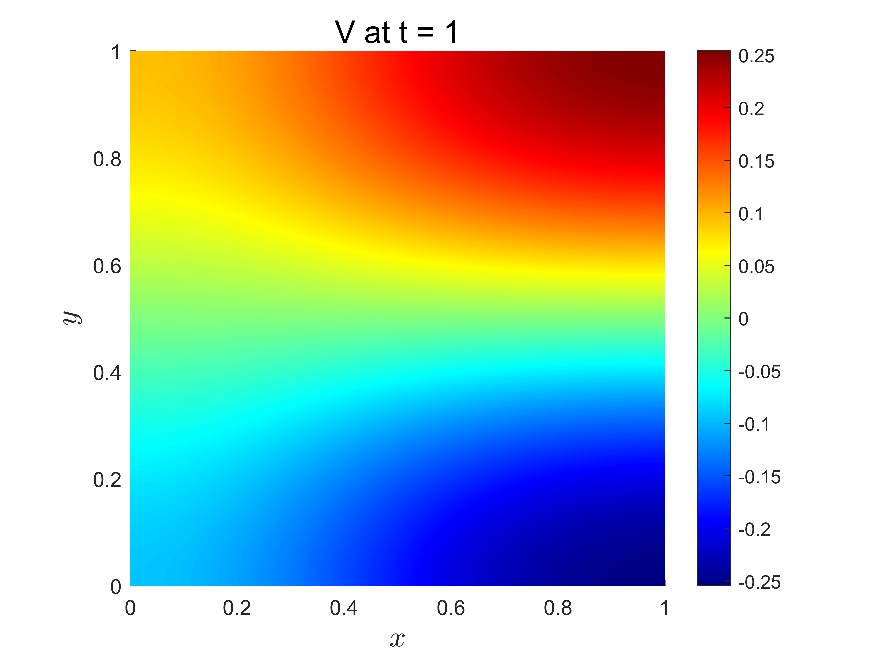}
\end{minipage}
\caption{Snapshots of $c_p$, $c_n$ and $V$ for $\omega_{11}=\omega_{22}=8$, $\omega_{12}=\omega_{21}=1$ at $t = 0.002$, $t = 0.1$ and $t = 1$.\label{StericEx3}}
\end{figure}

\begin{figure}[htb]
\setlength{\abovecaptionskip}{0pt}
\setlength{\abovecaptionskip}{0pt}
\setlength{\belowcaptionskip}{3pt}
\renewcommand*{\figurename}{Fig.}
\centering
\begin{minipage}[t]{0.29\linewidth}
\includegraphics[width=1.8in, height=1.6 in]{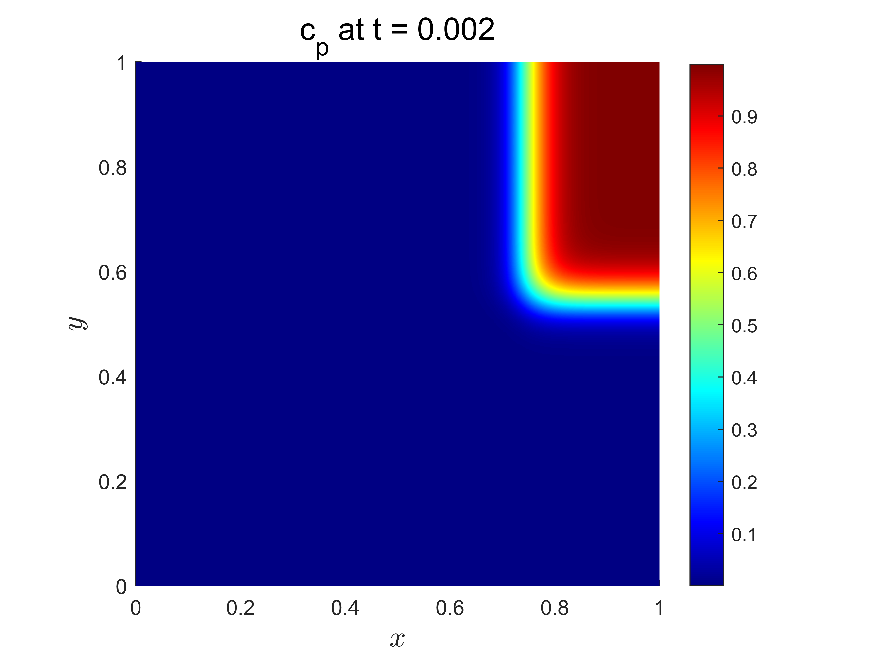}
\end{minipage}
\begin{minipage}[t]{0.29\linewidth}
\includegraphics[width=1.8in, height=1.6 in]{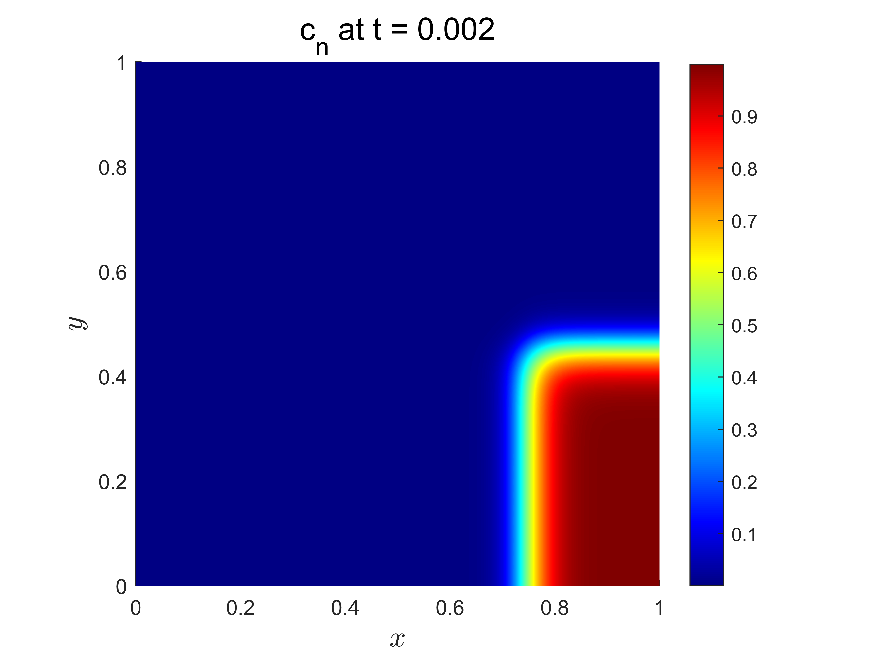}
\end{minipage}
\begin{minipage}[t]{0.29\linewidth}
\includegraphics[width=1.8in, height=1.6 in]{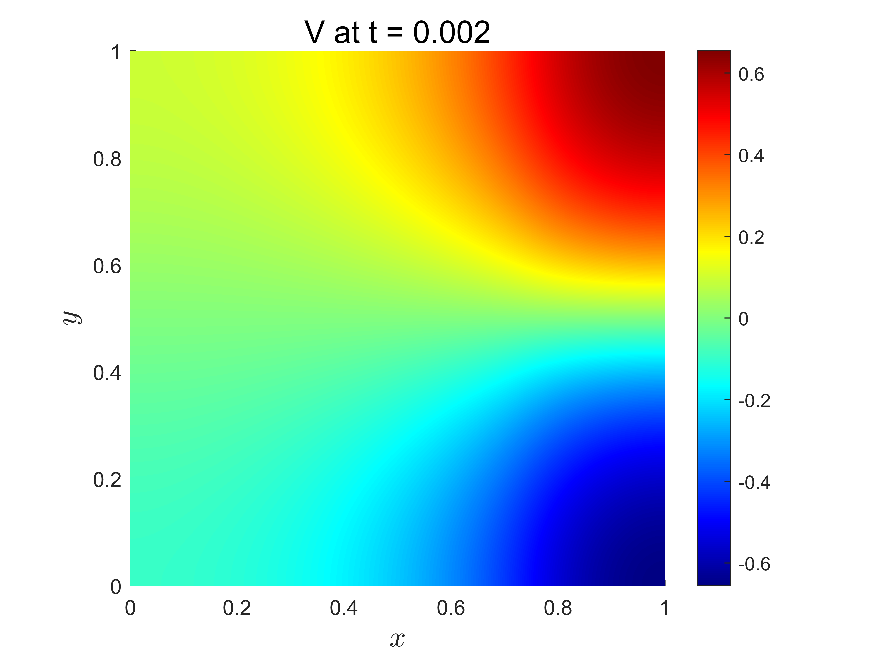}
\end{minipage}
\\
\begin{minipage}[t]{0.29\linewidth}
\includegraphics[width=1.8in, height=1.6 in]{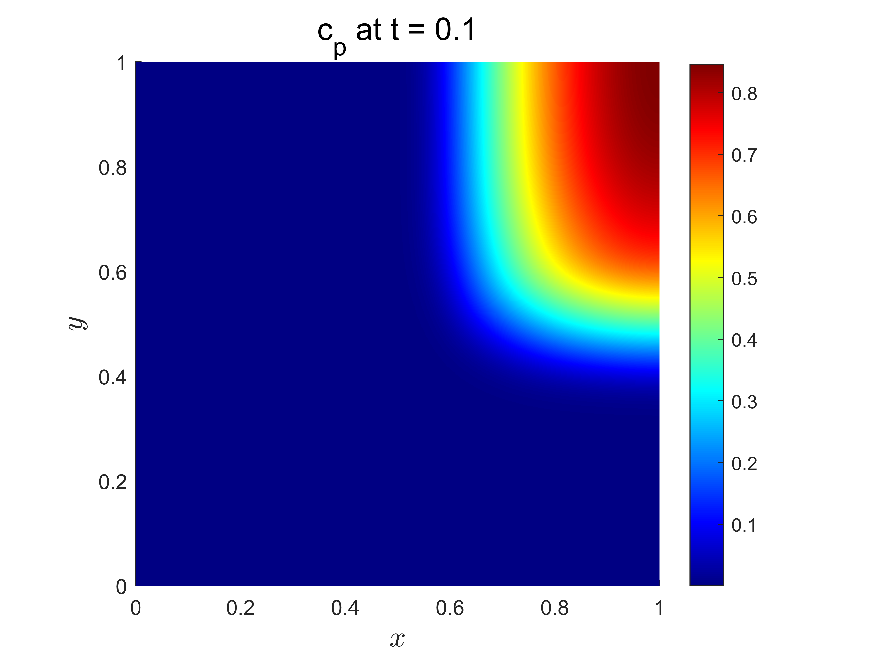}
\end{minipage}
\begin{minipage}[t]{0.29\linewidth}
\includegraphics[width=1.8in, height=1.6 in]{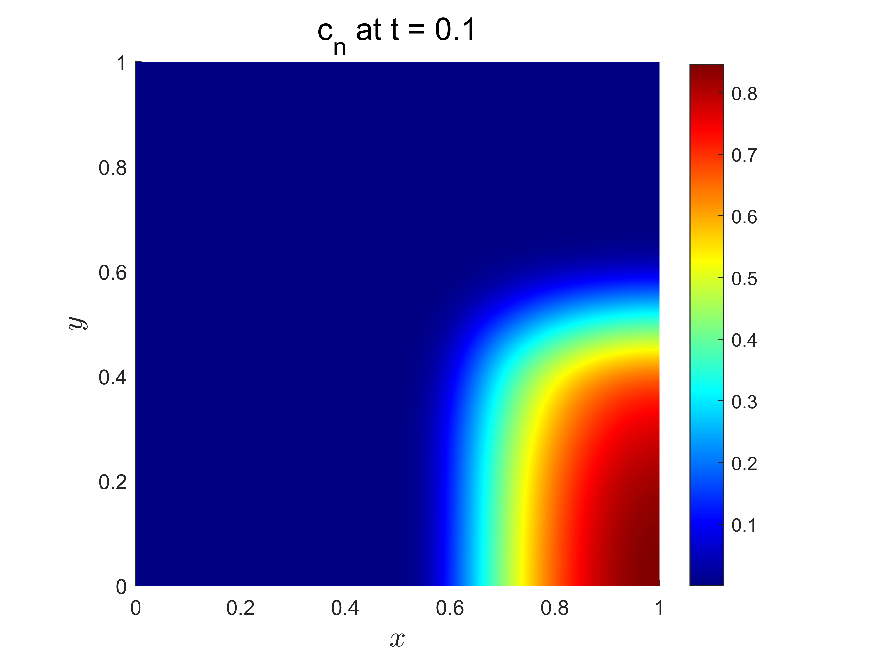}
\end{minipage}
\begin{minipage}[t]{0.29\linewidth}
\includegraphics[width=1.8in, height=1.6 in]{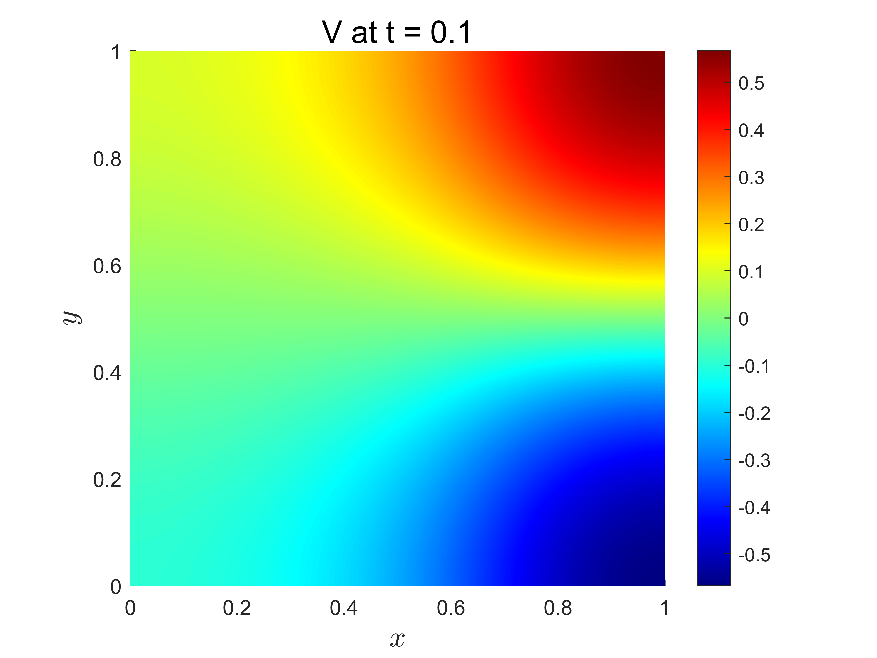}
\end{minipage}
\\
\begin{minipage}[t]{0.29\linewidth}
\includegraphics[width=1.8in, height=1.6 in]{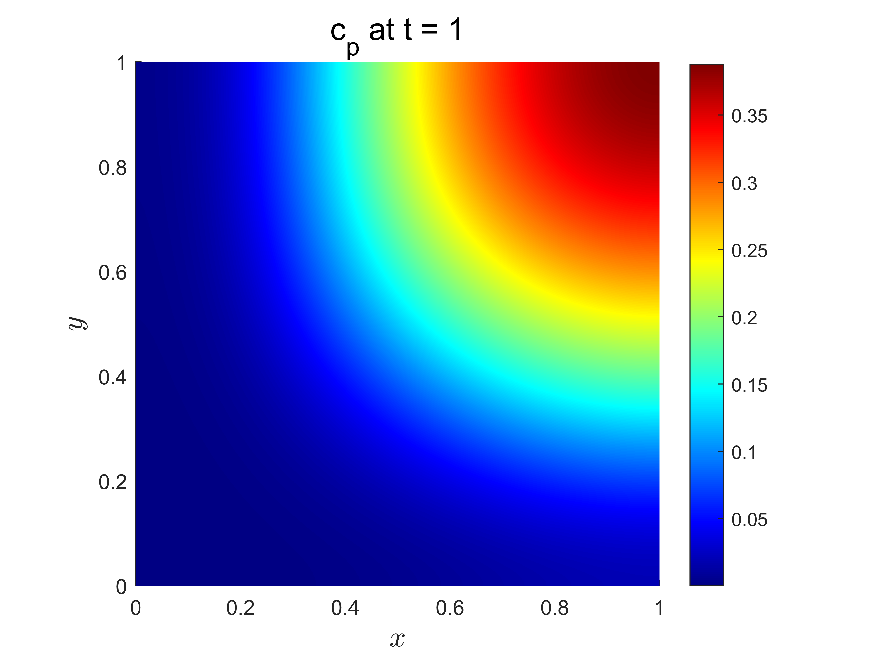}
\end{minipage}
\begin{minipage}[t]{0.29\linewidth}
\includegraphics[width=1.8in, height=1.6 in]{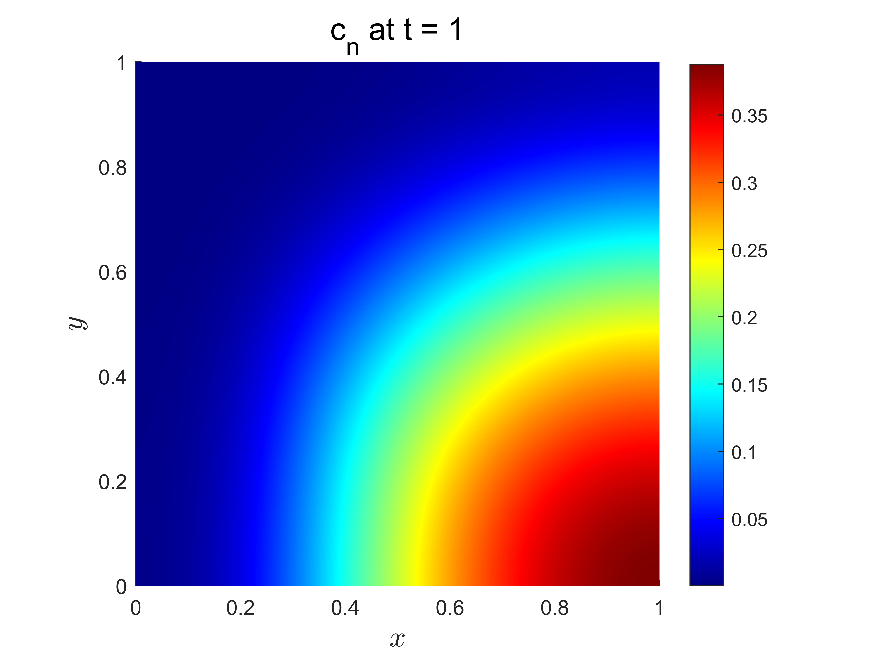}
\end{minipage}
\begin{minipage}[t]{0.29\linewidth}
\includegraphics[width=1.8in, height=1.6 in]{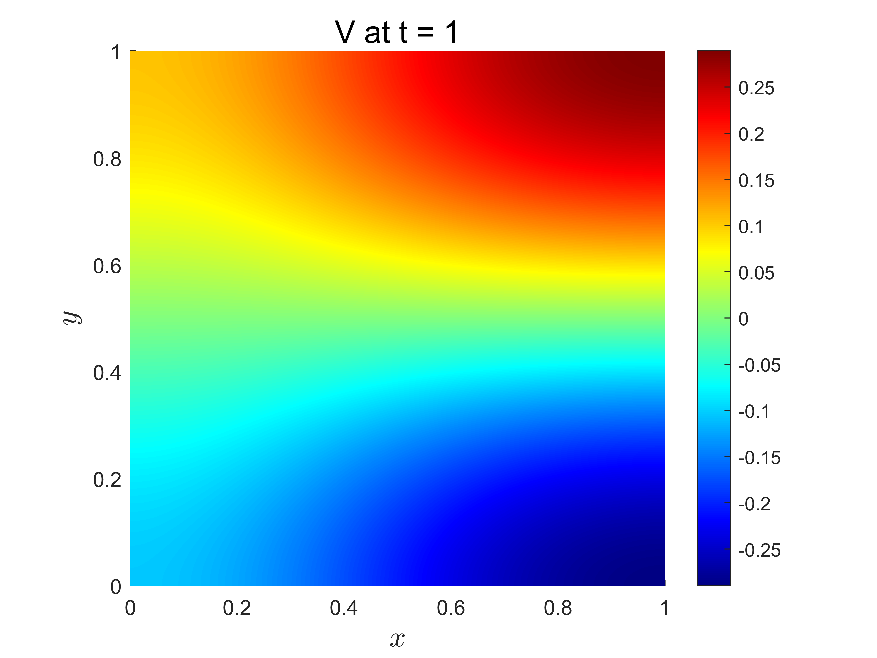}
\end{minipage}
\caption{Snapshots of $c_p$, $c_n$ and $V$ for $\omega_{11}=\omega_{22}=8$, $\omega_{12}=\omega_{21}=4$ at $t = 0.002$, $t = 0.1$ and $t = 1$.\label{StericEx4}}
\end{figure}

\begin{figure}[htb]
\setlength{\abovecaptionskip}{0pt}
\setlength{\abovecaptionskip}{0pt}
\setlength{\belowcaptionskip}{3pt}
\renewcommand*{\figurename}{Fig.}
\centering
\begin{minipage}[t]{0.29\linewidth}
\includegraphics[width=1.8in, height=1.6 in]{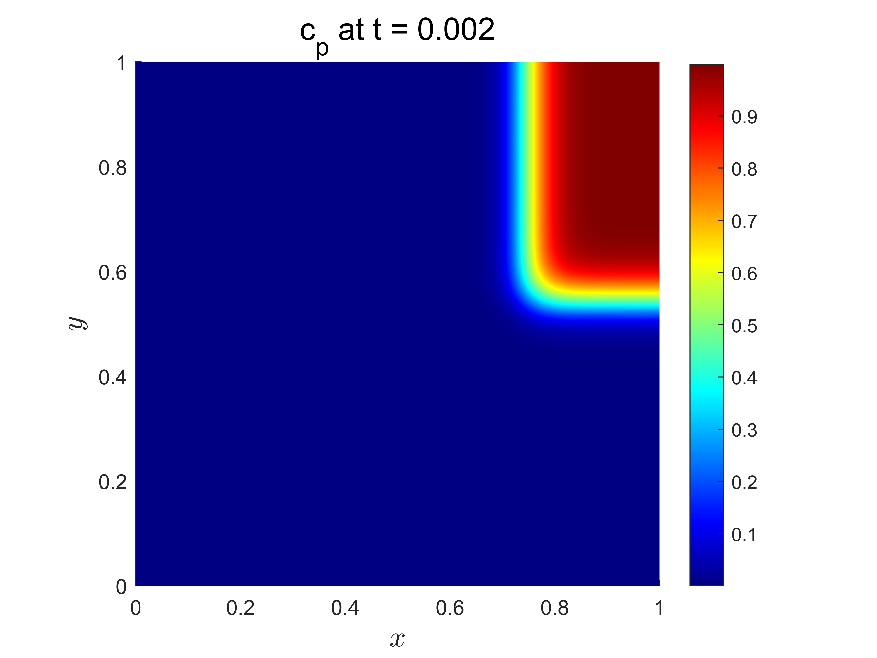}
\end{minipage}
\begin{minipage}[t]{0.29\linewidth}
\includegraphics[width=1.8in, height=1.6 in]{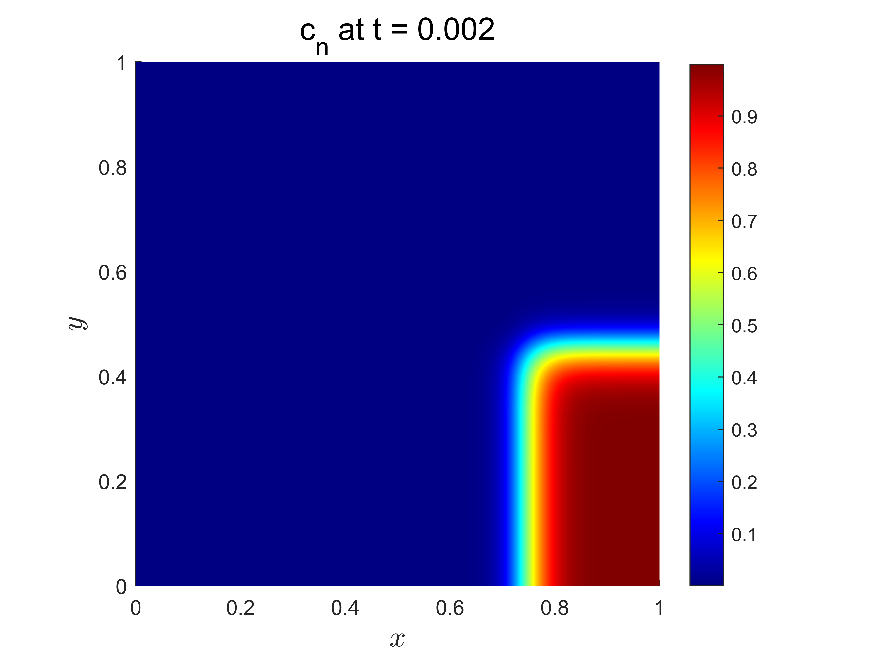}
\end{minipage}
\begin{minipage}[t]{0.29\linewidth}
\includegraphics[width=1.8in, height=1.6 in]{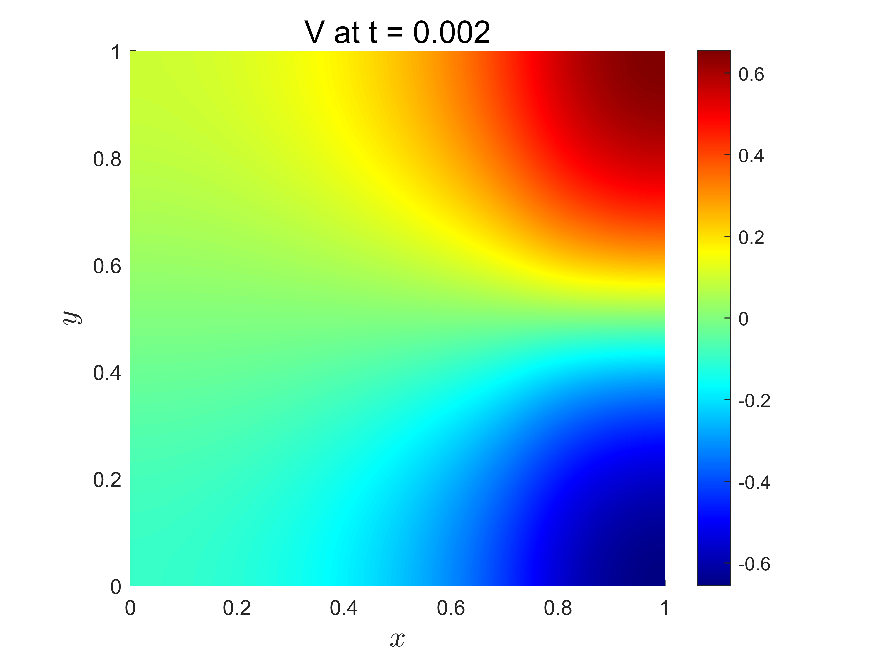}
\end{minipage}
\\
\begin{minipage}[t]{0.29\linewidth}
\includegraphics[width=1.8in, height=1.6 in]{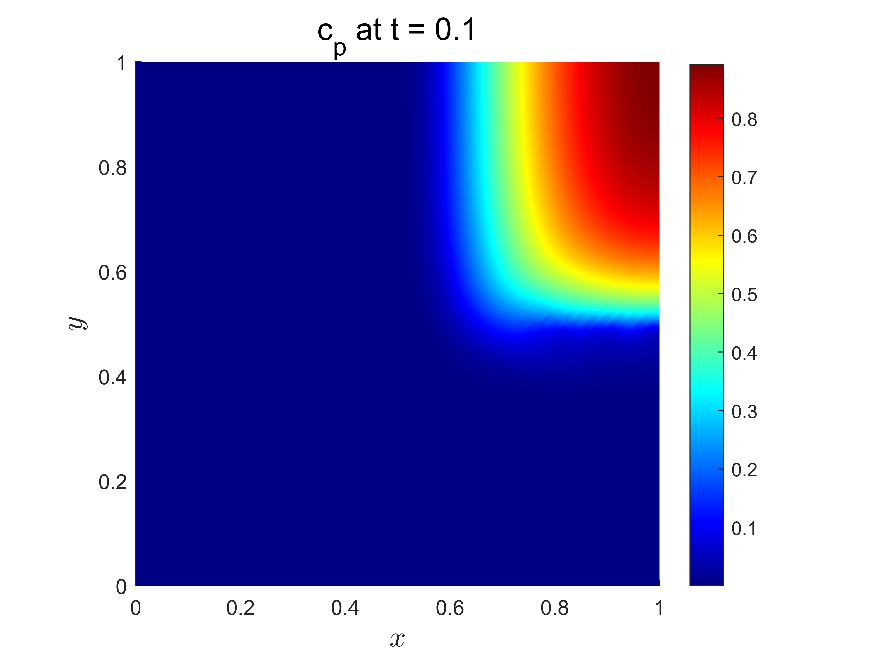}
\end{minipage}
\begin{minipage}[t]{0.29\linewidth}
\includegraphics[width=1.8in, height=1.6 in]{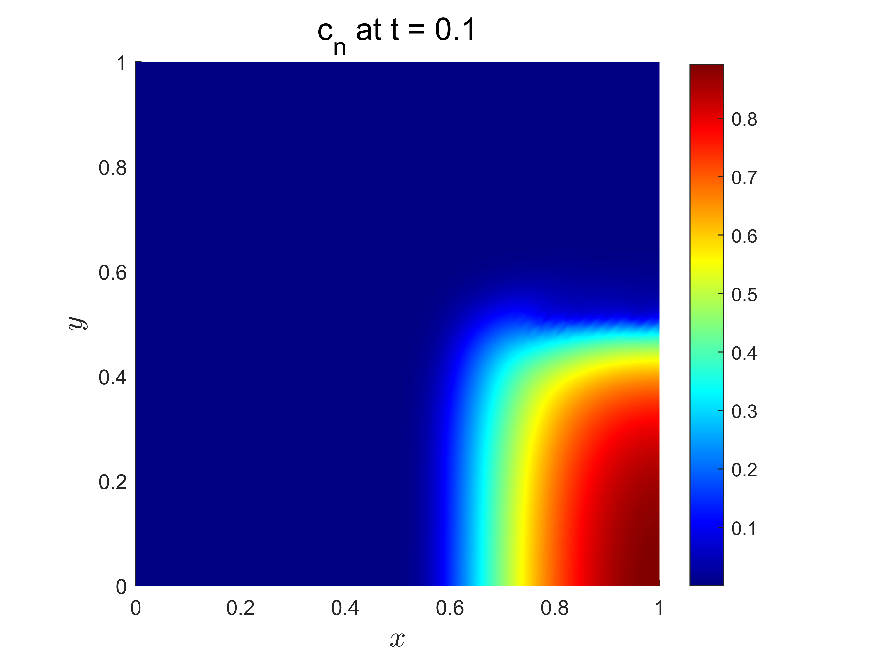}
\end{minipage}
\begin{minipage}[t]{0.29\linewidth}
\includegraphics[width=1.8in, height=1.6 in]{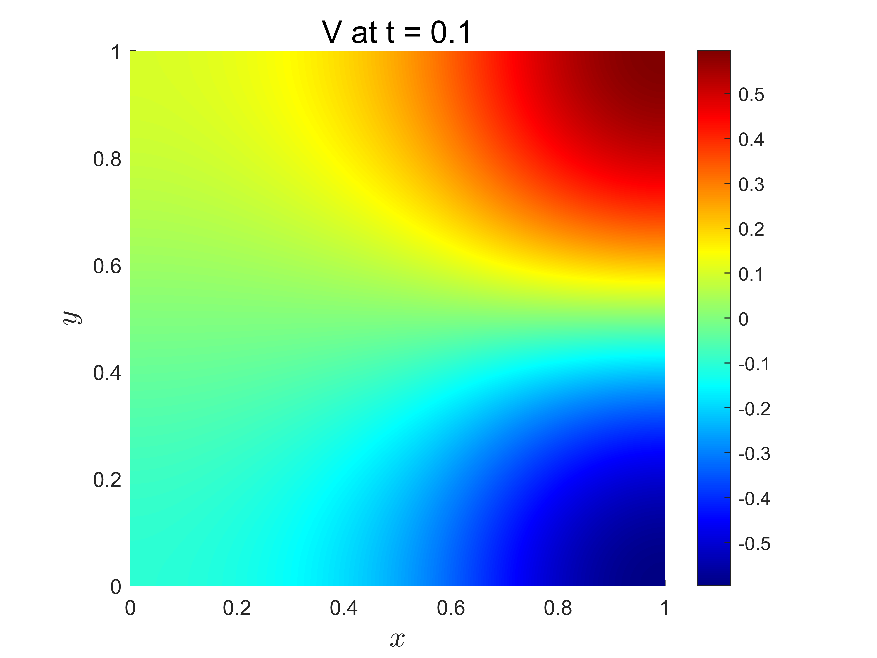}
\end{minipage}
\\
\begin{minipage}[t]{0.29\linewidth}
\includegraphics[width=1.8in, height=1.6 in]{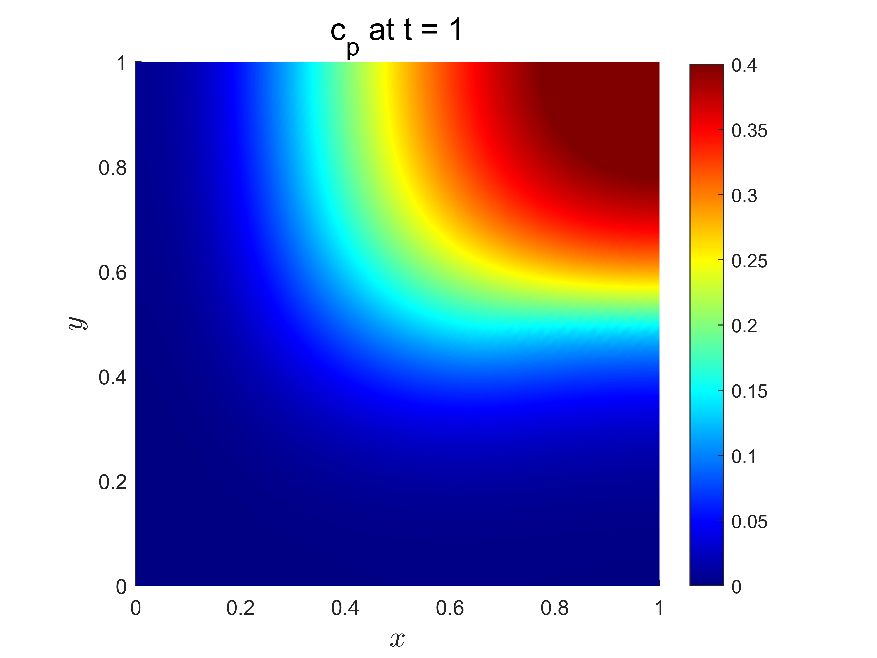}
\end{minipage}
\begin{minipage}[t]{0.29\linewidth}
\includegraphics[width=1.8in, height=1.6 in]{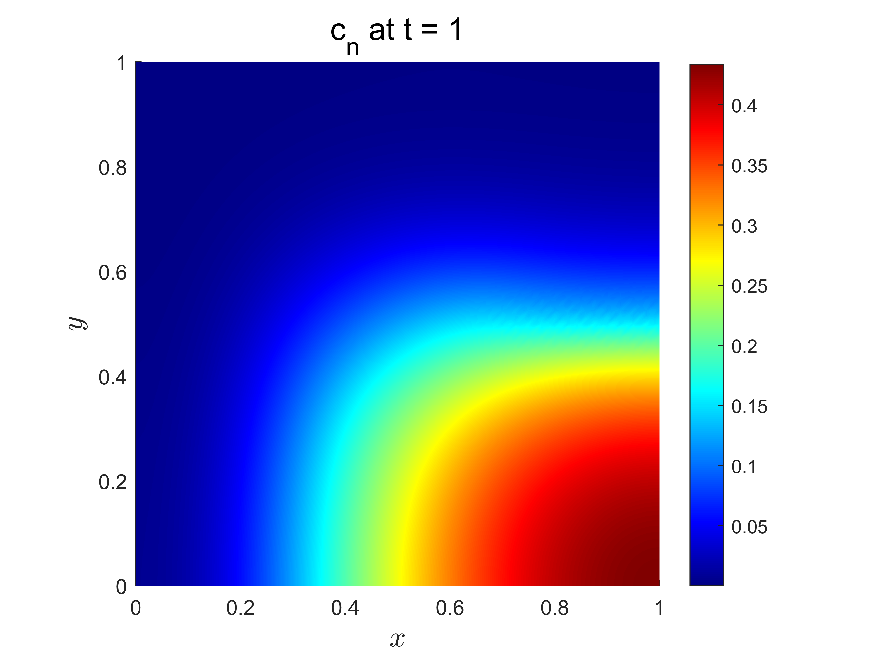}
\end{minipage}
\begin{minipage}[t]{0.29\linewidth}
\includegraphics[width=1.8in, height=1.6 in]{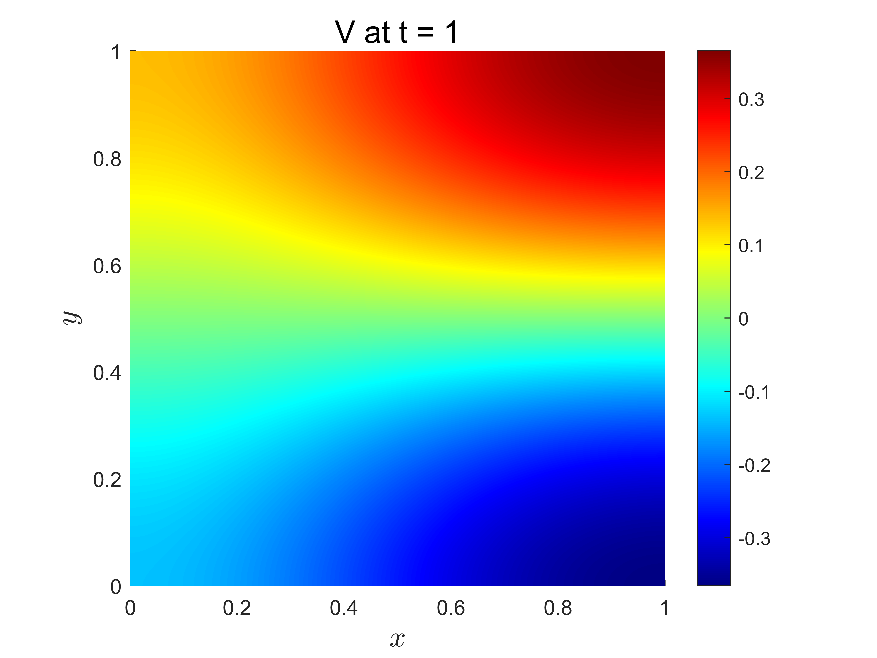}
\end{minipage}
\caption{Snapshots of $c_p$, $c_n$ and $V$ for $\omega_{11}=\omega_{22}=8$, $\omega_{12}=\omega_{21}=7$ at $t = 0.002$, $t = 0.1$ and $t = 1$.\label{StericEx4}}
\end{figure}

\subsection{Effects of exponent $k$}	
In this test, we consider the effects of exponent $k$ on the velocity field. We set the computed domain to be $\Omega=[0,1]^2$. The initial and boundary conditions are
\begin{equation}\label{test3:InitBC}
\begin{split}
& c_{p 0}=1+10^{-6}- \tanh( 100( (x-0.4)^2 + (y-0.4)^2 - (0.05)^2 ) ), \\
& c_{n 0}=1+10^{-6}- \tanh( 100( (x-0.6)^2 + (y-0.6)^2 - (0.05)^2 ) ), \quad \bm{u}_{0}=(0,0), \\
&x=0: \quad   V = 1,\quad\frac{\partial c_p}{\partial \mathbf{n}}=\frac{\partial c_n}{\partial \mathbf{n}}=0, \quad u=0, \quad v=0, \\
&x=1: \quad  V = 0,\quad\frac{\partial c_p}{\partial \mathbf{n}}=\frac{\partial c_n}{\partial \mathbf{n}}=0,\quad u=0, \quad v=0,\\
&y=0, 1 : \quad\frac{\partial V}{\partial \mathbf{n}}=\frac{\partial c_p}{\partial \mathbf{n}}=\frac{\partial c_n}{\partial \mathbf{n}}=0,  \quad u=0, \quad v=0.
\end{split}
\end{equation}
The parameters are set as $\lambda=0.1$, $Pe=50$, $Re=50$, $Co=100$, $\mu_0=1.0$, $\mu_{\infty}=0.1$, $\lambda_1=0.1$, $\Delta t=0.001$ and $h=\frac{\sqrt{2}}{60} $.
The time evolution of the streamlines are given in Fig. \ref{Ex4}.
Across all $k$ values, a single vortex was observed at $t=5$, with its shape distinctly influenced by the shear behavior. For $k<1$, indicative of shear thinning, the vortex core appeared flattened along the flow direction. At $k=1$, representing Newtonian behavior, the vortex core was less flattened. Conversely, for $k>1$, which denotes shear thickening, the vortex core was more rounded and oriented more vertically. Fig. \ref{Ex41} shows the snapshots of $c_p-c_n$ for $t = 0.005,0.075,0.1,0.2,0.3,5$ at $k=0.4$. At the end of the computation, the fluid becomes almost electron-neutral.

\begin{figure}[htb]
\setlength{\abovecaptionskip}{0pt}
\setlength{\abovecaptionskip}{0pt}
\setlength{\belowcaptionskip}{3pt}
\renewcommand*{\figurename}{Fig.}
\centering
\begin{minipage}[t]{0.29\linewidth}
\includegraphics[width=1.8in, height=1.5 in]{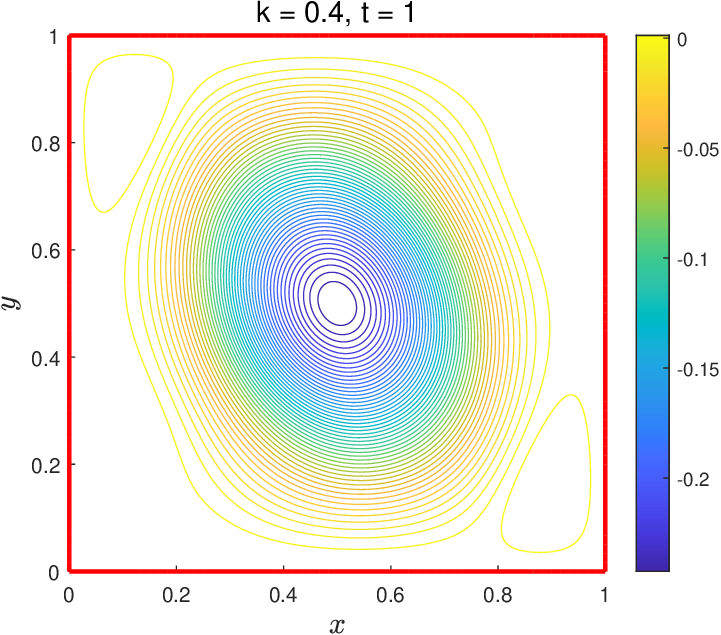}
\end{minipage}
\begin{minipage}[t]{0.29\linewidth}
\includegraphics[width=1.8in, height=1.5 in]{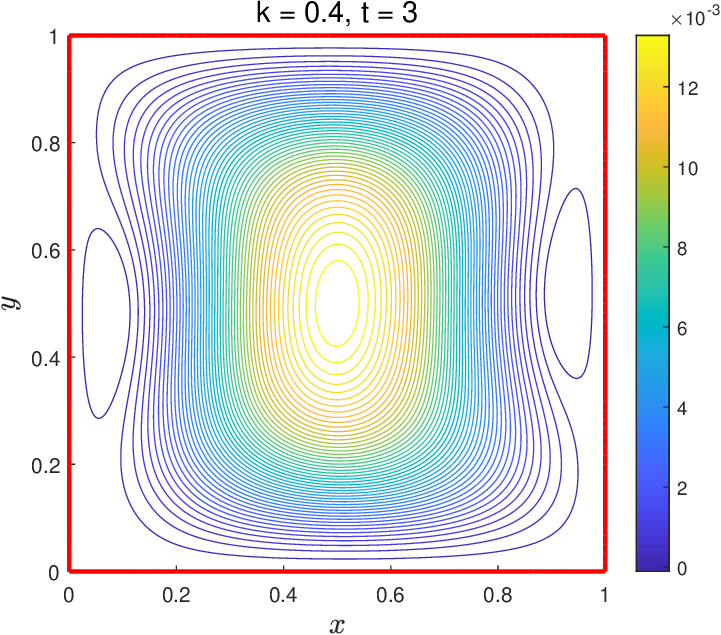}
\end{minipage}
\begin{minipage}[t]{0.29\linewidth}
\includegraphics[width=1.8in, height=1.5 in]{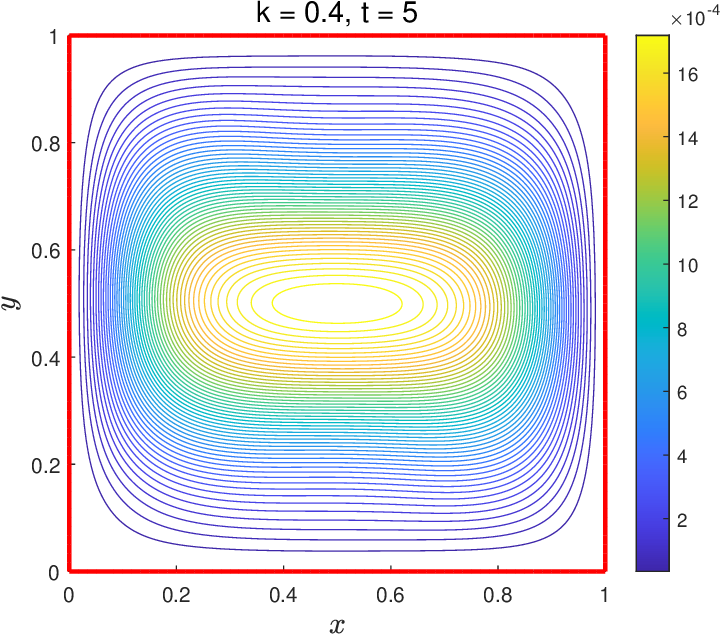}
\end{minipage}
\\
\begin{minipage}[t]{0.29\linewidth}
\includegraphics[width=1.8in, height=1.5 in]{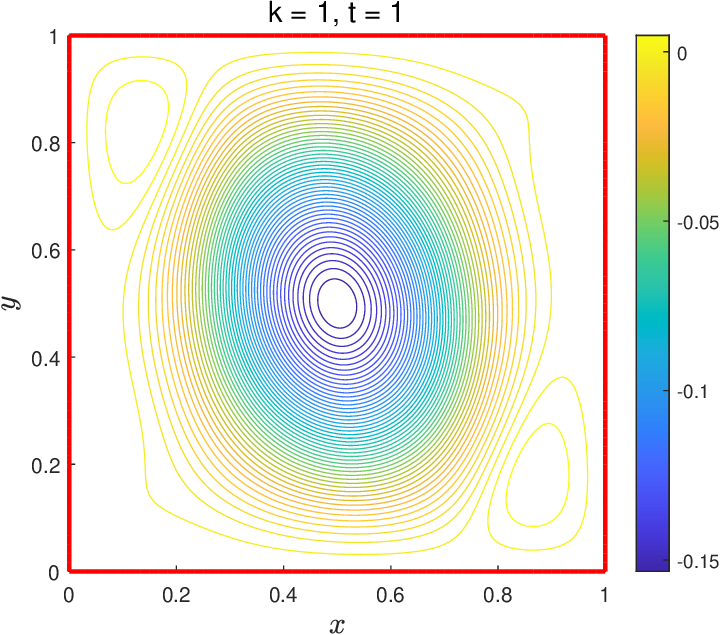}
\end{minipage}
\begin{minipage}[t]{0.29\linewidth}
\includegraphics[width=1.8in, height=1.5 in]{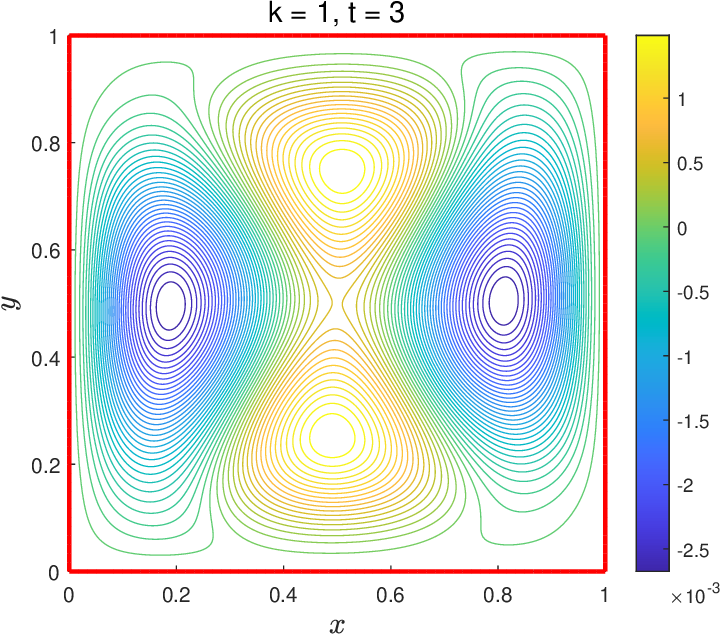}
\end{minipage}
\begin{minipage}[t]{0.29\linewidth}
\includegraphics[width=1.8in, height=1.5 in]{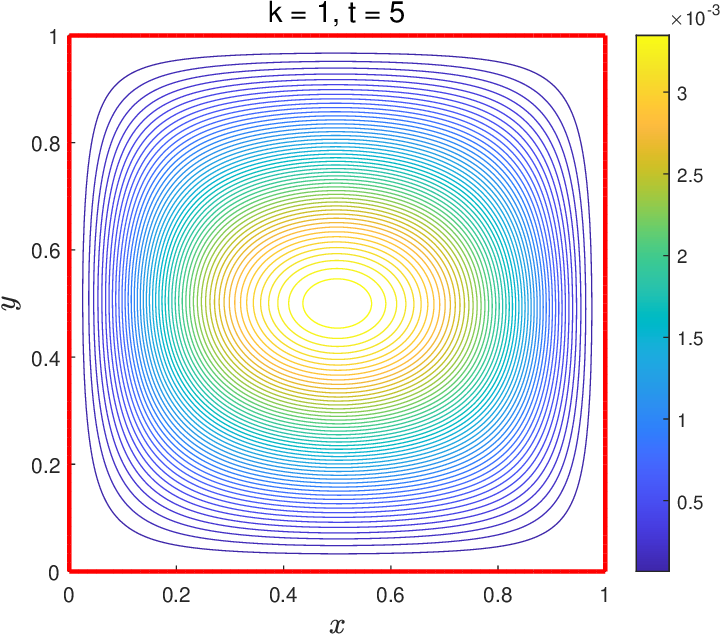}
\end{minipage}
\\
\begin{minipage}[t]{0.29\linewidth}
\includegraphics[width=1.8in, height=1.5 in]{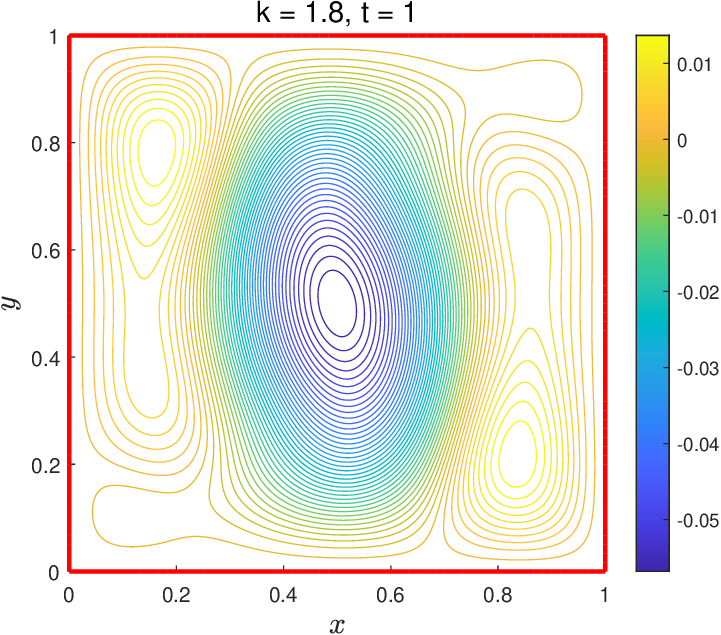}
\end{minipage}
\begin{minipage}[t]{0.29\linewidth}
\includegraphics[width=1.8in, height=1.5 in]{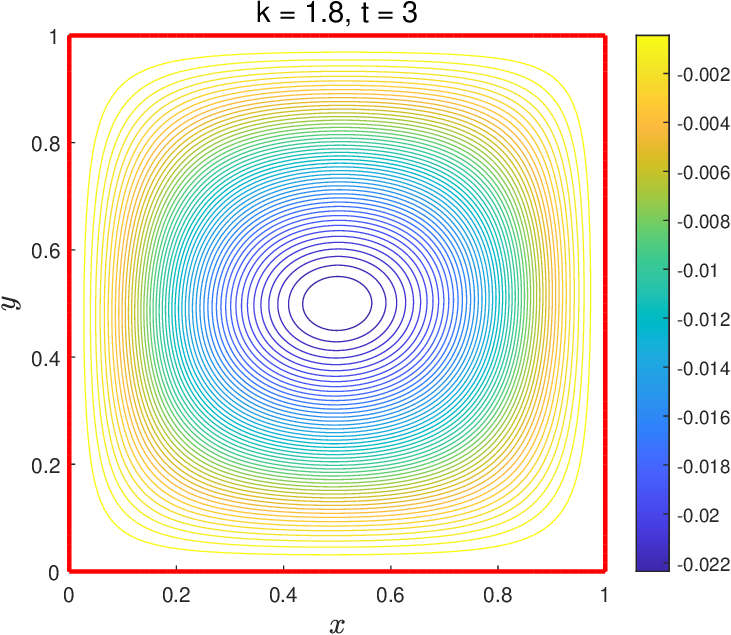}
\end{minipage}
\begin{minipage}[t]{0.29\linewidth}
\includegraphics[width=1.8in, height=1.5 in]{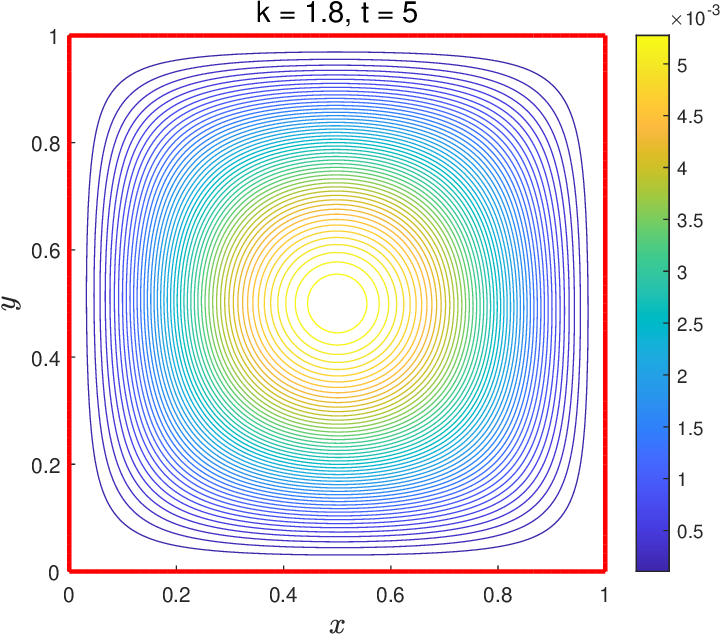}
\end{minipage}
\caption{Time evolution of streamlines for $k = 0.4$, $k = 1$ and $k = 1.8$.\label{Ex4}}
\end{figure}

\begin{figure}[htb]
\setlength{\abovecaptionskip}{0pt}
\setlength{\abovecaptionskip}{0pt}
\setlength{\belowcaptionskip}{3pt}
\renewcommand*{\figurename}{Fig.}
\centering
\begin{minipage}[t]{0.29\linewidth}
\includegraphics[width=1.8in, height=1.6 in]{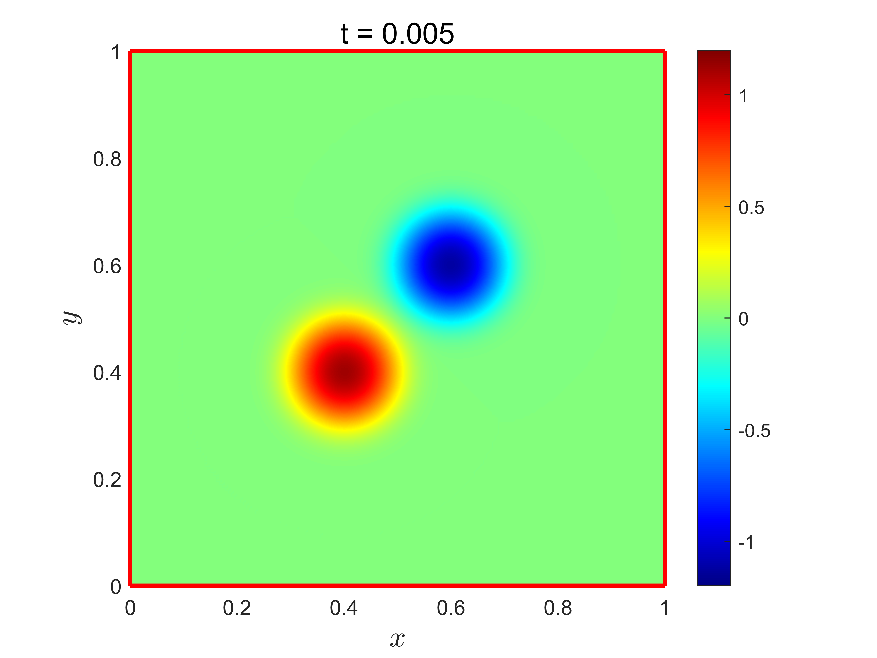}
\end{minipage}
\begin{minipage}[t]{0.29\linewidth}
\includegraphics[width=1.8in, height=1.6 in]{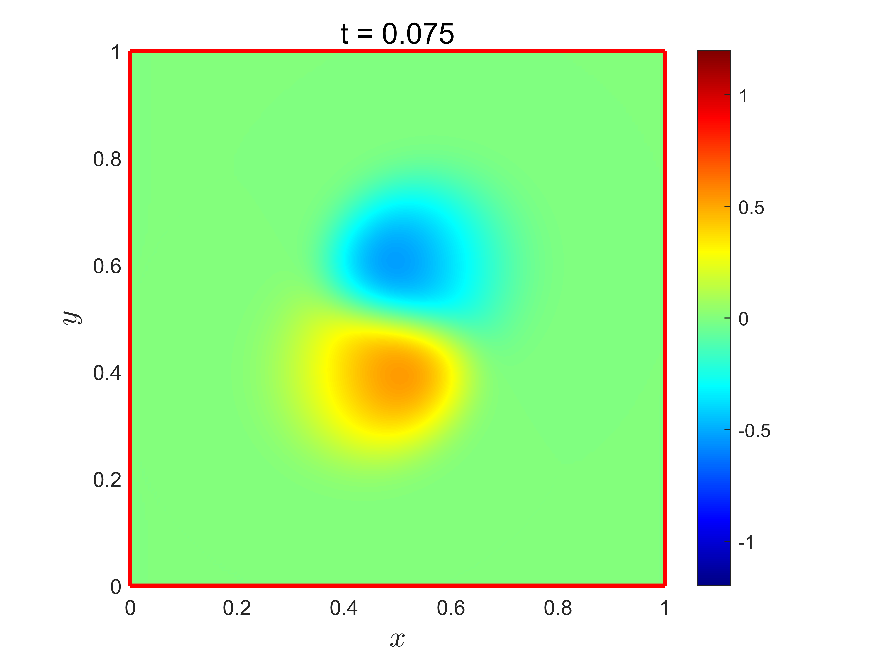}
\end{minipage}
\begin{minipage}[t]{0.29\linewidth}
\includegraphics[width=1.8in, height=1.6 in]{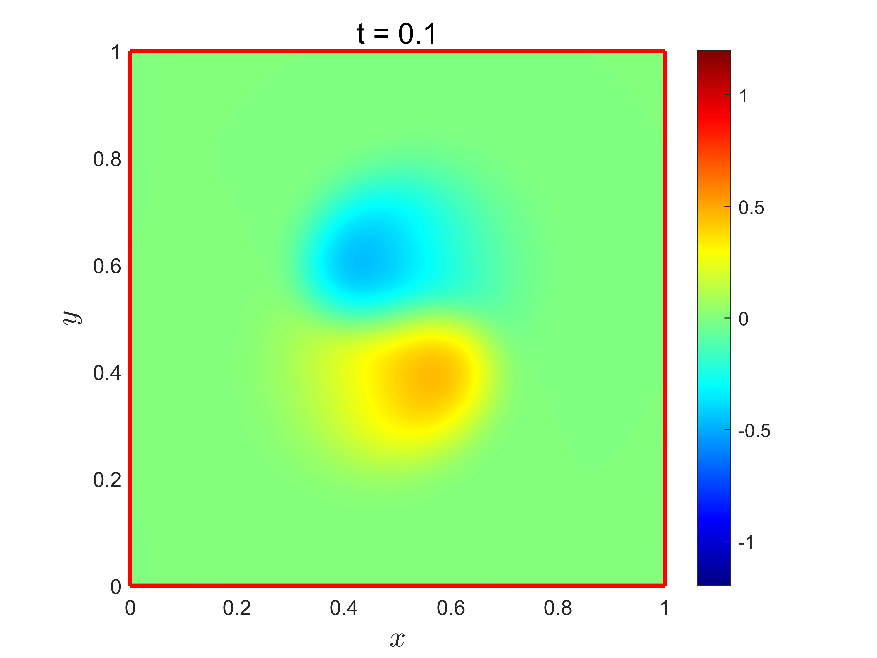}
\end{minipage}
\\
\begin{minipage}[t]{0.29\linewidth}
\includegraphics[width=1.8in, height=1.6 in]{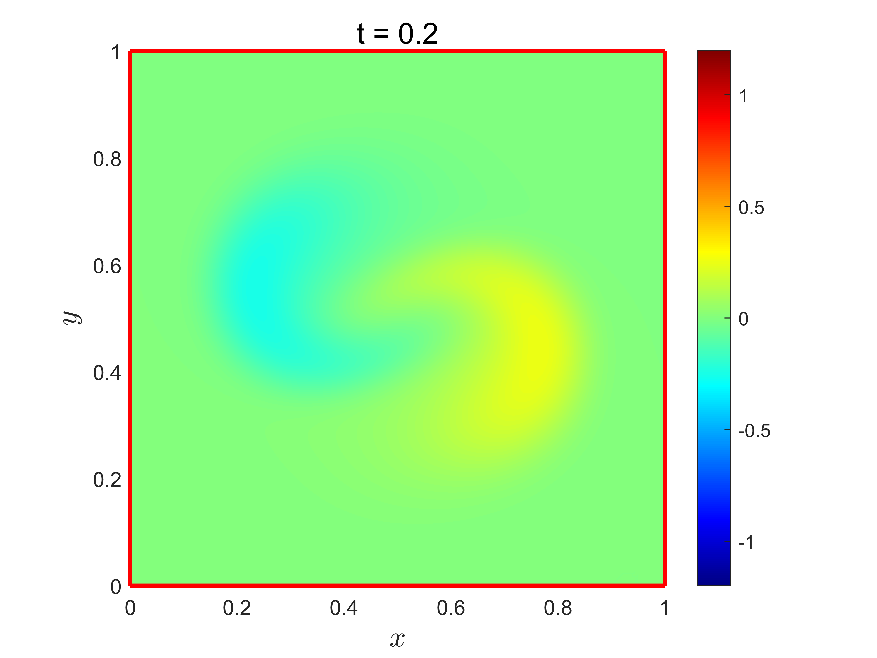}
\end{minipage}
\begin{minipage}[t]{0.29\linewidth}
\includegraphics[width=1.8in, height=1.6 in]{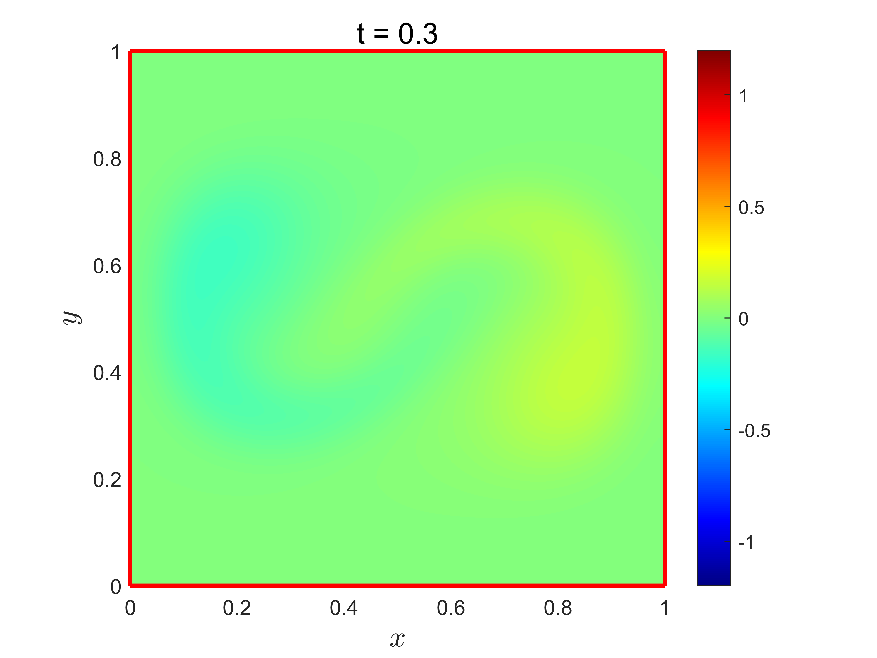}
\end{minipage}
\begin{minipage}[t]{0.29\linewidth}
\includegraphics[width=1.8in, height=1.6 in]{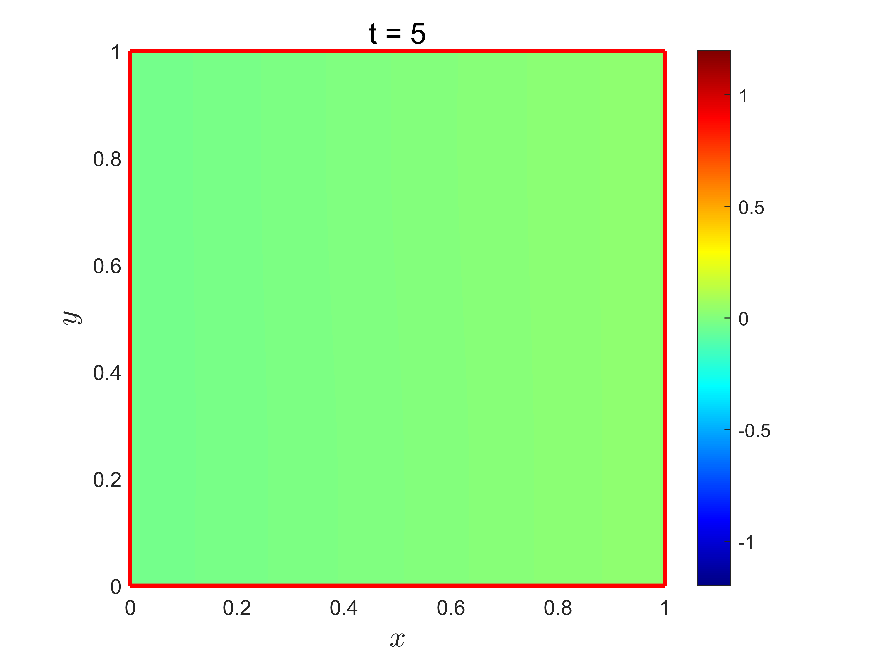}
\end{minipage}
\caption{Snapshots of $c_p-c_n$ for $t = 0.005,0.075,0.1,0.2,0.3,5$ at $k=0.4$.\label{Ex41}}
\end{figure}

\section{Conclusions}
In this paper, to explore the impact of steric effect on ion transport in non-Newtonian fluid, we have developed a linear, decoupled, second-order accurate in time and structure-preserving scheme for Carreau fluid equations coupled with steric PNP model. The novelty of the proposed scheme is based on the decoupling technique by introducing a nonlocal auxiliary variable with respect to the free energy of steric PNP equations and splitting method for the fluid equation. Moreover, at each time step, we only need to solve several linear equations. The fully discrete numerical scheme has proved to be energy stable, and preserve positivity and mass conservation of the ion concentration. The accuracy and effectiveness of the scheme are verified by numerical experiments.

\bibliographystyle{unsrt}
\bibliography{reference}
	
\end{document}